\documentclass[10pt]{elsarticle-ppn}

\makeatletter
\let\c@author\relax
\makeatother

\usepackage{fancyhdr}
\usepackage{comment}
\usepackage{marginnote}
\usepackage{booktabs}
\usepackage{multirow}
\usepackage[scr=rsfs,cal=boondox]{mathalfa}
\usepackage{soul}

\usepackage[backend=biber, citestyle=numeric-comp, bibstyle=trad-abbrv, url=false, doi=false, isbn=false]{biblatex}
\usepackage{amsmath}
\usepackage{mathtools}
\usepackage{amsfonts}
\usepackage{amsthm}
\usepackage{amssymb}
\usepackage{color}
\usepackage{dsfont}
\usepackage{geometry}
\usepackage{todonotes}
\usepackage{yhmath}
\usepackage[normalem]{ulem}
\newgeometry{tmargin=2.8cm, bmargin=3.5cm, lmargin=1.7cm, rmargin=1.7cm}

\usepackage{mathrsfs}
\usepackage{dutchcal}
\usepackage{hyperref}
\usepackage{cleveref}

\newcommand{\R}{{\mathbb{R}}}
\newcommand{\rn}{{\mathbb{R}^N}}
\newcommand{\Rm}{{\mathbb{R}^m}}

\newcommand{\N}{{\mathbb{N}}}

\newcommand{\betapow}{{\beta}}

\newcommand{\ve}{{\varepsilon}}

\newcommand{\vk}{{\varkappa}}

\newcommand{\Cb}{{C_\sigma}}
\newcommand{\Ca}{{C_a}}
\newcommand{\Mbound}{{K}}
\def\d{{\,{\rm d}}}

\newcommand{\Cnum}{{C}}

\newcommand{\tildec}{{\tilde c}}
\newcommand{\bu}{{f}}
\newcommand{\Uset}{{U}}
\newcommand{\bfun}{{\sigma}}
\newcommand{\wfun}{{\omega}}
\newcommand{\Aform}{{\Lambda}}
\newcommand{\Omprim}{{\Omega'}}
\newcommand{\ai}{{s}}
\newcommand{\aii}{{s}}
\newcommand{\bii}{{\tau}}

\newcommand{\cZ}{{\mathcal{Z}}}

\newcommand{\cF}{{\mathcal{F}}}

\newcommand{{\SP}}{{\mathcal{Z}}}
\newcommand{\dnum}{{d}}
\newcommand{\anyopen}{{U}}
\newcommand{\Ai}{{\lambda}}

\DeclareMathOperator{\supp}{supp}

\DeclareMathOperator{\dist}{dist}
\DeclareMathOperator{\diam}{diam}

\def\Xint#1{\mathchoice
{\XXint\displaystyle\textstyle{#1}}%
{\XXint\textstyle\scriptstyle{#1}}%
{\XXint\scriptstyle\scriptscriptstyle{#1}}%
{\XXint\scriptscriptstyle\scriptscriptstyle{#1}}%
\!\int}
\def\XXint#1#2#3{{\setbox0=\hbox{$#1{#2#3}{\int}$ }
\vcenter{\hbox{$#2#3$ }}\kern-.6\wd0}}

\def\dashint{\Xint-}

\usepackage{refcount}

\newcounter{cte}

\definecolor{darkgreen}{rgb}{0.00, 0.50, 0.00}

\definecolor{mahogany}{rgb}{0.75, 0.25, 0.00}

\def\Xint#1{\mathchoice
   {\XXint\displaystyle\textstyle{#1}}%
   {\XXint\textstyle\scriptstyle{#1}}%
   {\XXint\scriptstyle\scriptscriptstyle{#1}}%
   {\XXint\scriptscriptstyle\scriptscriptstyle{#1}}%
   \!\int}
\def\XXint#1#2#3{{\setbox0=\hbox{$#1{#2#3}{\int}$}
     \vcenter{\hbox{$#2#3$}}\kern-.5\wd0}}

\def\dashint{\Xint-}

\def\YYint#1#2#3{{\setbox0=\hbox{$#1{#2#3}{\iint}$}
    \vcenter{\hbox{$#2#3$}}\kern-.51\wd0}}

\theoremstyle{plain}
\newtheorem{proposition}{Proposition}[section]
\newtheorem{lemma}[proposition]{Lemma}
\newtheorem{theorem}{Theorem}
\newtheorem{corollary}[proposition]{Corollary}

\theoremstyle{remark}

\theoremstyle{definition}

\numberwithin{equation}{section}

\allowdisplaybreaks

\begin{document}
\begin{frontmatter}

\title{Smoothness of weight sharply discards\\ Lavrentiev's gap for double phase functionals%\footnote{moze cos jak: Lavrentiev Gap for Double Phase Functional Sharply Defeated by Weight Smoothness}
}

\author{Michał Borowski}\ead{m.borowski@mimuw.edu.pl}

\address{Institute of Applied Mathematics and Mechanics,
University of Warsaw, ul. Banacha 2, 02-097 Warsaw, Poland
}

\fntext[myfootnote]{Mathematical Subject Classification: 49J45 (46E30, 46E40, 46A80)}
\fntext[myfootnote]{The author is supported by the Polish Ministry of Science and Education grant PN/02/0001/2023.}

\begin{abstract}
    We show that the smoother the weight, the broader the range of exponents for which the Lavrentiev's gap is absent for the double phase functionals, i.e.,
    \begin{equation*}
        u \mapsto \int_{\Omega} \left(|\nabla u|^p + a(x)|\nabla u|^q\right)\,\d x\,, \quad 1 \leq p \leq q < \infty,\, a(\cdot) \geq 0\,.
    \end{equation*}
    In particular, if $a \in C^\infty$, then no additional restrictions are required on $p$ and $q$.
    For $a \in C^{k, \alpha}$, we establish the optimal range of exponents, which reads $q \leq p + (k + \alpha)\max(1, p/N)$. Thereby, we extend previously known results which consider H\"older continuous $a$ (i.e., $q \leq p + \alpha\max(1, p/N)$), showing that the range of exponents extends naturally upon imposing more regularity on $a$.
\end{abstract}

\end{frontmatter}

\section{Introduction and main results}
We study the minimization problem of the double phase functionals, i.e., 
\begin{equation}\label{eq:defF}
    \cF[u;\Omega] \coloneqq \int_{\Omega} \left(|\nabla u|^p + a(x)|\nabla u|^q\right)\,\d x\,,
\end{equation}
where a priori $1 \leq p \leq q < \infty$, the weight function $a$ is non-negative and bounded, and $\Omega \subseteq \rn$ is a bounded Lipschitz domain. We provide a new technique that leads to a conclusion that if the weight $a$ belongs to the class $C^{k, \alpha}$ and if $q \leq p + (k + \alpha)\max(1, p/N)$, then the Lavrentiev's phenomenon for functional~\eqref{eq:defF} is excluded. This is the content of Theorem~\ref{theo:Cka}. It is essential to note that this result is sharp due to the counterexamples already present in the literature, such as the ones in~\cite{eslemi, badi}. We also point out that $a \in C^{\infty}$ excludes Lavrentiev's phenomenon for all exponents $p$ and $q$, as stated in Corollary~\ref{coro:Cinf}.

Previously, the result had been proven only for $k = 0$, by~\cite{bgs-arma, eslemi}, and $k = 1$, by~\cite{BCFM}. The methods therein strongly rely on the analysis of the rate of vanishing of the weights, which does not apply to $k>1$. We expose this in more detail in the further part of the introduction. To handle weights smoother than $C^{1,\alpha}$, we introduce an innovative decomposition of the weight and utilize the theory of Muckenhoupt $A_r$ classes. Initially, we prove the absence of Lavrentiev's gap for weights that admit the factorization involving the Muckenhoupt class and a class of weights with a sufficient decay rate. %\blue{We stress that such a decomposition is introduced here for the first time.}
This is described by Theorem~\ref{theo:Lavdecomp}. It turns out that weights of suitable smoothness rate always admit the prescribed factorization --- a surprising fact that we prove, see Theorem~\ref{theo:decomp}.

Our results provide a foundation for possible further development of the theory of regularity for double phase functionals with regular weights, such as the regularity of minimizers in the spirit of~\cite{comi, colmar, HaOk1, HaOk2, BaaBy, bacomi-cv, Besch, ChCh, Ok}.\newline

{\noindent \bf Lavrentiev's phenomenon.} One of the challenges when dealing with functionals of $(p, q)$-growth, including~\eqref{eq:defF}, is to address the so-called Lavrentiev's phenomenon. For us, this means that we would like to exclude the situation in which
\begin{equation}\label{eq:LavPres}
    \inf_{u \in u_0 + W^{1, 1}_0(\Omega)} \cF[u] < \inf_{u \in u_0 + C_c^{\infty}(\Omega)} \cF[u]
\end{equation}
for any given $u_0 \in W^{1, 1}(\Omega)$, playing the role of boundary data, such that $\cF[u_0] < \infty$. Inequality~\eqref{eq:LavPres} means that minimizers of the functional $\cF$ cannot be approximated in the energy by smooth functions, which typically precludes any regularity theory, as seen in~\cite{BadMinim}. This is why the important first step in studying double phase functionals is to identify those parameters for which Lavrentiev's gap is excluded. This topic for functional~\eqref{eq:defF} and other problems resembling double phase structure has been widely studied, e.g., in~\cite{BaaBy, basu, badi, filomena, CarMeDeLi, zh, eslemi, BCFM}. Recently, the topic of excluding Lavrentiev's phenomenon for general classes of functionals was developed, see~\cite{BCM, bgs-arma, BMT, BBCLM, ha-aniso, Lukas23, Bousquet-Pisa}. In parallel, developments have been made to the related problem of establishing density of smooth functions in Musielak-Orlicz-Sobolev spaces, see~\cite{yags, hahab, Bor-Chl, ACCZG, C-b, Gossez}. For more information on Lavrentiev's phenomenon and related topics, we refer to survey articles~\cite{Buttazzo-Belloni, CerfMar, Wiktor, MinRad, SomeLav, Pocket}. \newline

{\noindent \bf Double phase functionals and prior results.} The study of the functional~\eqref{eq:defF} dates back to~\cite{zh} and is one of the most profound examples of the functionals with so-called $(p, q)$-growth, introduced in~\cite{Marcellini}. The functional models the switch between $p$- and $q$-growth regimes (phases) according to the weight $a$. The set where $a$ vanishes corresponds to the $p$-phase, its complement to the $p$-$q$-phase, while the decay or regularity of $a$ is expected to govern the transition between those phases.

Studying functional~\eqref{eq:defF} presents typical challenges that appear in studying $(p, q)$-growth problems. It turns out that one cannot hope for any regularity of minimizers unless the exponents $p$ and $q$ are sufficiently close to each other, with closeness quantified by the behavior of the weight $a$, see~\cite{BadMinim}. Initial idea, developed through~\cite{eslemi, zh}, is to consider $a \in C^{0, \alpha}$ with $\frac{q}{p} \leq 1 + \frac{\alpha}{N}$. In particular, this excludes Lavrentiev's gap. The regularity theory under this condition has been established in the seminal paper by Colombo and Mingione~\cite{comi}. In parallel, the regularity theory of a priori bounded minimizers was also developed in~\cite{colmar}, this time in the range $q \leq p + \alpha$. The absence of Lavrentiev's gap in this case is established in~\cite{bgs-arma}. The related regularity results for minimizers with other a priori assumptions are given in~\cite{bacomi-cv, Ok}.

The conditions mentioned above have their drawbacks. For instance, for a given $p \leq N$, one a priori requires that $q \leq p +1$, as $C^{0, \alpha}$ is a non-trivial class only for $\alpha \in (0, 1]$. In~\cite{BCFM}, class $C^{0, \alpha}$ is replaced with $\cZ^{\alpha}$ for arbitrary $\alpha > 0$. Following~\cite[Definition 1.1]{BCFM}, we say that non-negative function $\bfun \in \SP^{\vk}(\anyopen)$, $\anyopen \subseteq \rn$ open, if there exists a constant $C$ such that for all $x, y \in \anyopen$, it holds
\begin{equation}\label{eq:defZ}
    \bfun(x) \leq C\left(\bfun(y) + |x-y|^{\vk} \right)\,.
\end{equation}
The class $\cZ^\vk$ only restricts how $\bfun$ approaches $0$, and imposes no control on $\bfun$ when it is separated from $0$. The parameter $\vk$ indicates how rapidly the function $\bfun$ decays. As shown in~\cite{BCFM}, weights of class $\cZ^{\vk}$ allow excluding Lavrentiev's phenomenon for an arbitrary choice of $p$ and $q$. Regularity of minimizers in this case comes from the general theory developed in~\cite{HaOk1, HaOk2}.

As described in~\cite[Proposition 1.3]{BCFM}, the class $\cZ^{\vk}$ does not share many commonalities with smoothness in general. Although the inclusions $C^{0, \alpha} \subset \cZ^{\alpha}$ and $C^{1, \alpha} \subset \cZ^{1 + \alpha}$ hold, there is no such inclusion for $C^{k, \alpha}$ with $k > 1$. Indeed, one can think about the function $t \mapsto t^2$, which, although being $C^{\infty}$, does not have sufficient decay rate to belong to $\cZ^{\vk}$ for any $\vk > 2$.
\newline

{\noindent \bf Main result.} A much more natural way of extending the previous results than involving the class $\cZ^\vk$ is to consider smoother weights, namely, of class $C^{k, \alpha}$, and expand the range of exponents so that the role of $\alpha$ in previous conditions is replaced with $k + \alpha$. As described in the last paragraph, such a result cannot be obtained for any $k > 1$ by the approach of~\cite{BCFM}, which generalizes those of~\cite{eslemi, bgs-arma}. Neither can it be obtained by applying any of the many conditions for discarding Lavrentiev's gap that are available in the literature, listed above, in the section about Lavrentiev's phenomenon. The present paper presents the first general results. 

In the following theorem, we prove the absence of Lavrentiev's gap under the assumptions that $a \in C^{k, \alpha}$ and $q \leq p + (k+\alpha)\max(1, p/N)$, for any $k \in \N$ and $\alpha \in (0, 1]$.
\begin{theorem}[Absence of Lavrentiev's phenomenon]\label{theo:Cka}
    Let $\Omega \subseteq \rn$ be a bounded Lipschitz domain and let $\Uset \subseteq \rn$ be open and such that $\Omega \Subset \Uset$. Let $\cF$ be defined as in~\eqref{eq:defF} with $q \leq p + (k+\alpha)\max(1, p/N)$ and non-negative $a \in C^{k, \alpha}(U)$ for some numbers $k \in \N$, $\alpha \in (0, 1]$. Then, for any $u_0 \in W^{1, 1}(\Omega)$ such that $\cF[u_0;\Omega] < \infty$, it holds
        \begin{equation}\label{eq:LavAbs}
            \inf_{u \in u_0 + W^{1, 1}_0(\Omega)} \cF[u] = \inf_{u \in u_0 + C_c^{\infty}(\Omega)} \cF[u]\,.
        \end{equation}
\end{theorem}
In Section~\ref{sec:furres}, we present a more general version of Theorem~\ref{theo:Cka}, namely, Corollary~\ref{coro:genapp}. It says that any function $u$ such that $\cF[u] < \infty$ may be suitably approximated by regular functions, not only $u$ being a minimizer. Moreover, we prove that the range of exponents may be further expanded if we a priori assume that $u \in C^{0, \gamma}$. This is a result in the spirit of~\cite{bacomi-cv, BCM}, which shows that the higher regularity of minimizers to~\eqref{eq:defF} that are a priori H\"older continuous may be proven under an improved range of exponents. \newline

{\noindent \bf Sharpness of Theorem~\ref{theo:Cka}.} In~\cite[Theorem 34]{basu}, it is proven that whenever $q > p + (k+\alpha)\max\left( 1, \tfrac{p-1}{N-1} \right)$, there exists $a \in C^{k, \alpha}$ such that Lavrentiev's phenomenon occurs, i.e.,~\eqref{eq:LavAbs} does not hold. If $p \leq N$, this range reads $q > p + k + \alpha$, while the range guaranteeing absence of Lavrentiev's phenomenon in Theorem~\ref{theo:Cka} reads $q \leq p + k + \alpha$. We thus have the sharp threshold of smoothness of weights beyond which Lavrentiev's gap is discarded, and before which it is not, in general. In this sense, classes $C^{k, \alpha}$ induce a precise scale to the problem.

In case of $p > N$, the aforementioned thresholds for $q$ do not agree, as they read $p + (k+\alpha)\tfrac{p}{N}$ and $p + (k+\alpha)\tfrac{p-1}{N-1}$, respectively. It is nonetheless expected that the former is optimal in the sense described above. Moreover, one can think about the condition from Theorem~\ref{theo:Cka} as the condition for the ratio of exponents, i.e., $\tfrac{q}{p} \leq 1 + \frac{k+\alpha}{N}$. In case of $k=0$, this range is generally acknowledged to be sharp, due to a counterexample in~\cite[Section 3]{eslemi}. In~\cite[Section 4]{BCFM}, this counterexample is generalized to all $k \in \N$. Although the goal is to provide $a \in \cZ^{k+\alpha}$, one can verify that the weight $a$ used there is of class $C^{k, \alpha}$. \newline

Let us point out an important corollary of Theorem~\ref{theo:Cka}.
\begin{corollary}[No Lavrentiev's gap for $C^{\infty}$ weights] \label{coro:Cinf}
    Let $\Omega \subseteq \rn$ be a bounded Lipschitz domain and $\Uset \subseteq \rn$ be open and such that $\Omega \Subset \Uset$. Let $\cF$ be defined as in~\eqref{eq:defF} with any $1 \leq p \leq q < \infty$ and non-negative $a \in C^{\infty}(\Uset)$. Then, for any $u_0 \in W^{1, 1}(\Omega)$ such that $\cF[u_0;\Omega] < \infty$, it holds
        \begin{equation*}
            \inf_{u \in u_0 + W^{1, 1}_0(\Omega)} \cF[u] = \inf_{u \in u_0 + C_c^{\infty}(\Omega)} \cF[u]\,.
        \end{equation*}
\end{corollary}
{\noindent \bf Factorization of weights.} Theorem~\ref{theo:Cka} is a consequence of two results --- Theorems~\ref{theo:Lavdecomp} and~\ref{theo:decomp}. Theorem~\ref{theo:Lavdecomp} provides the absence of Lavrentiev's phenomenon for the class of weights admitting a certain decomposition related to $\cZ^{\vk}$ and Muckenhoupt classes. On the other hand, Theorem~\ref{theo:decomp} says that such a decomposition is always admitted by $C^{k, \alpha}$ functions. Both of those theorems are new.

Before presenting the precise statements, let us introduce the (local) Muckenhoupt weights, as defined in~\cite{extrap}. For every $r \in (1, \infty)$ and open set $\anyopen$, we say that a function $\wfun : \anyopen \to [0, \infty)$ belongs to class $A_r(\anyopen)$ if
\begin{equation}\label{eq:defMuck}
    \sup_{B \subseteq \anyopen} \left( \dashint_{B} \wfun \d z \right) \left( \dashint_{B} \wfun^{-\frac{1}{r-1}} \d z\right)^{r-1} < \infty\,,
\end{equation}
where supremum is taken over all balls inside $\anyopen$. Muckenhoupt classes appear naturally as the class of weights guaranteeing boundedness of the maximal operator in weighted Lebesgue spaces --- a crucial property when proving approximation results via mollification.

We are now in a position to state the following result.
\begin{theorem}\label{theo:Lavdecomp}
        Let $\Omega \subseteq \rn$ be an open and bounded Lipschitz domain and let $\Uset$ be an open set such that $\Omega \Subset \Uset$. Recall the definitions of classes $\cZ^\vk$ and $A_r$ from~\eqref{eq:defZ} and~\eqref{eq:defMuck}. Let $\cF$ be defined as in~\eqref{eq:defF} for numbers $p, q$ and non-negative function $a$ satisfying $a = \bfun \wfun$, where $\bfun \in \SP^{\vk}(\Omega)$ for $1 \leq p \leq q \leq p + \vk \max(1, p/N)$ and $\wfun \in A_q(\Uset) \cap L^\infty(\Omega)$. Then, for any $u_0 \in W^{1, 1}(\Omega)$ such that $\cF[u_0;\Omega] < \infty$, it holds
        \begin{equation*}
            \inf_{u \in u_0 + W^{1, 1}_0(\Omega)} \cF[u] = \inf_{u \in u_0 + C_c^{\infty}(\Omega)} \cF[u]\,.
        \end{equation*}
\end{theorem}
The statement (but not the proof) of Theorem~\ref{theo:Lavdecomp} has been inspired by the results in~\cite{ZhSur}. This work also considers a product of the weights of two separate classes, although of other kinds and in a different context.

We provide a more general version of Theorem~\ref{theo:Lavdecomp} later, i.e., Theorem~\ref{theo:genapprox}. In contrast to other results available in the literature, such as~\cite[Theorem 3]{BCFM}, Theorem~\ref{theo:Lavdecomp} allows for considering weights of the form of a product of $\cZ^{\vk}$-function, $\vk > 0$, (including $C^{0, \vk}$ if $\vk \in (0, 1]$) and a bounded weight satisfying a Muckenhoupt-type condition. At first glance, it is not clear how widely Theorem~\ref{theo:Lavdecomp} extends the previous contributions. Surprisingly, Theorem~\ref{theo:Lavdecomp} may be directly applied to prove Theorem~\ref{theo:Cka}. The following theorem justifies this fact.
\begin{theorem}\label{theo:decomp}
    Let $\Omega$ and $\Uset$ be open and bounded subsets of $\rn$ such that $\Omega \Subset \Uset$. Recall the definitions of classes $\cZ^\vk$ and $A_r$ from~\eqref{eq:defZ} and~\eqref{eq:defMuck}. Let $a : \Uset \to [0, \infty)$ be of class $C^{k, \alpha}$. Then, there exist functions $\bfun, \wfun : \Omega \to [0, \infty)$ such that $\bfun \in \SP^{k+\alpha}$, $\wfun \in A_{k+\alpha+1}(\Omega) \cap L^{\infty}(\Omega)$ and $a = \bfun \wfun$.
\end{theorem}
Theorem~\ref{theo:Lavdecomp} suggests the possibility of extending the entire theory of double phase functionals to the one with weights proposed therein by incorporating the Muckenhoupt weights theory into the field. On the other hand, Theorem~\ref{theo:decomp} shows a vital application of considering the weights as in Theorem~\ref{theo:Lavdecomp}, namely, for smoother weights than of H\"older continuity class. We thus believe that Theorems~\ref{theo:Lavdecomp} and~\ref{theo:decomp} can be of crucial use for further developments of the theory of double phase functionals with $C^{k, \alpha}$ weights.

Let us stress that the decomposition proposed in Theorem~\ref{theo:decomp} is in general non-trivial. First, let us point our attention to the functions $x \mapsto |x_1|^{2n}$, $n \in \N$, which although being $C^{\infty}$, belong to class $\cZ^{\vk}$ only for $\vk \leq 2n$, while belong to the class $A_{\vk + 1}$ only for $\vk > 2n$. Both components $\bfun$ and $\wfun$ are thus important. Second, note that the function $x \mapsto e^{-\frac{1}{|x|^2}}$ is $C^\infty$, belongs to $\cZ^{\vk}$ with any $\vk > 0$, but fails to satisfy the Muckenhoupt condition with any finite parameter. Consequently, it is impossible to embed any smoothness class into some Muckenhoupt class. Finally, consider function $t \mapsto t^6\sin^2(1/t)$. This function is of class $C^{2, 1}$, but does not belong to either $\cZ^3$ nor $A_4$, even after restriction to any neighbourhood of $0$. However, it may be proved that $t \mapsto t^4$ belongs to $\cZ^3$, while $t \mapsto t^2\sin^2(1/t)$ belongs to $A_4$. Therefore, it is not only false that $a$ always coincides with either $\bfun$ or $\wfun$, but it is also impossible to point out separate regions of $\Omega$ where $a$ behaves more like $\bfun$ or $\wfun$. We remark that the factorization of Theorem~\ref{theo:decomp} trivializes only in the cases of $k = 0$ and $k = 1$, as we have $C^{0, \alpha} \subseteq \cZ^{\alpha}$ and $C^{1, \alpha} \subseteq \cZ^{1 + \alpha}$, see~\cite[Remark 1.2 and Proposition 1.3]{BCFM}.

We also stress that although the results of Theorems~\ref{theo:Cka} and~\ref{theo:Lavdecomp} are local, due to the requirement that the weight lives on a bigger set than $\Omega$, it is ultimately not a setback for the further investigation of minimizers. For instance, works~\cite{eslemi, comi, colmar} use only local approximation results to prove (local) regularity of minimizers to~\eqref{eq:defF}.\newline

{\noindent \bf Open question on regularity of minimizers.} The regularity theory for the minimizers of double phase functionals, in terms of belonging to classes $C^{0, \beta}, C^{1, \beta}$, already received notable attention when considering H\"older continuous weights, see~\cite{eslemi, colmar, comi, bacomi-cv, Ok}. Later works of~\cite{HaOk1, HaOk2, hhl}, dealing with far more general problems, in particular establish regularity theory for double phase with weights in class $\cZ^{\vk}$. We stress that the mentioned papers do not cover the case of weights of class $C^{k, \alpha}$ in general, and thus, the suitable regularity theory may still be developed. We believe that Theorems~\ref{theo:Lavdecomp} and~\ref{theo:decomp} are a vital step in this direction. Indeed, they suggest an approach to the problem that combines two fields of great interest: double phase functionals and problems in weighted Sobolev spaces. The latter one is already well established, as described, for instance, in~\cite{NLP, Fabes}.

We also emphasize that, as our results do not align with previously known results, there is still room for improvement in the general conditions for discarding Lavrentiev's gap, for which our result can be a starting point. Moreover, our results have consequences for approximation properties in Musielak-Orlicz-Sobolev spaces of double phase type, and for related PDEs, which we briefly describe in Section~\ref{sec:Musielak} \newline  

{\noindent \bf Methods.} The first idea of the proof of Theorem~\ref{theo:Cka} is to join two different approaches to the problem. The first approach, developed through~\cite{zh, eslemi, bgs-arma} and culminating in~\cite{BCFM}, involves utilizing the rate of decay of the weight, expressed via the class $\SP$. The second one, which we introduce into the topic, is to use the density of smooth functions in the weighted Sobolev space associated with the weight of the (local) Muckenhoupt $A_q$ class. Neither method alone suffices to exploit the higher smoothness of $a$. It turns out that both approaches can be merged in the sense of Theorem~\ref{theo:Lavdecomp}, i.e., if the weight decomposes into the product of weights, each one being suitable for a different approach. Still, it is not obvious Theorem~\ref{theo:Lavdecomp} may be applied to smooth weights. This proved to be the most challenging part of the present paper, which we justify in Theorem~\ref{theo:decomp}.

To prove it, we need to construct a suitable decomposition of any $C^{k, \alpha}$ function into a product of $\cZ^{\vk}$ and Muckenhoupt weights. The first key observation is that although a smooth weight itself does not have to admit a sufficient rate of decay to belong to a suitable $\cZ^{\vk}$ class, it can be forced to do so via corrections by its derivatives, as seen in Lemma~\ref{lem:Zpart}. The proof that the remaining factor is of a suitable Muckenhoupt class is the content of Proposition~\ref{prop:Muck}. In the core, the approach relies on a careful, local comparison of the weight to its Taylor polynomial, also taking advantage of its non-negativity through techniques similar to the so-called Glaeser inequality. This allows not only checking the Muckenhoupt condition on small balls, but also asserting much more precise control on the balls we call medium-sized. The latter one makes it possible to use covering arguments in order to check the Muckenhoupt condition on a large scale, which would not be possible with the sole information that the Muckenhoupt condition is satisfied on the balls from the covering. For the more detailed outline of the proof, see the beginning of Section~\ref{sec:proofs1}. 

The proofs of Theorem~\ref{theo:Lavdecomp} and its generalization, Theorem~\ref{theo:genapprox}, employ techniques already well-established in the field. The proof of Theorem~\ref{theo:genapprox} may be viewed as an extension of the proof of~\cite[Theorem 3]{BCFM}. The additional steps account for the Muckenhoupt part of the weight, which we use for the boundedness of the maximal operator. \newline

{\noindent \bf Organization.} In Section~\ref{sec:furres}, we present further results and applications of our main theorems. In particular, in Section~\ref{sec:gen}, we present Theorem~\ref{theo:genapprox} and Corollary~\ref{coro:genapp}, generalizing Theorems~\ref{theo:Cka} and~\ref{theo:Lavdecomp}. Section~\ref{sec:proofs1} is devoted to the proof of Theorem~\ref{theo:decomp}, while in Section~\ref{sec:proofs2}, we give proofs of Theorems~\ref{theo:Cka},~\ref{theo:Lavdecomp} and~\ref{theo:genapprox}. \newline

{\noindent \bf Notation.} For every $x \in \rn$ and $r > 0$, we shall denote a ball centred in $x$ and with radius $r$ by $B(x, r)$ or $B_r(x)$. For any ball $B = B(x, r)$ and any $c > 0$, we shall denote $cB \coloneqq B(x, cr)$.

For any function $f \in C^{0, \beta}$, $\beta \in (0, 1]$, we shall denote by $[f]_{C^{0, \beta}}$ the H\"older seminorm of $f$.

Throughout the paper, we shall use the letter $c$ to denote a positive constant that may differ from line to line. Sometimes, to set the attention, we shall indicate in parentheses what constant $c$ depends on.
\section{Further results}\label{sec:furres}
\subsection{General approximation results}\label{sec:gen}
We now provide a general approximation theorem (Theorem~\ref{theo:genapprox}), which extends Theorem~\ref{theo:Lavdecomp} to arbitrary finite-energy functions. Namely, we prove that any $u \in W^{1, 1}_0(\Omega)$ may be approximated by $C_c^{\infty}(\Omega)$ in a way related to functional~\eqref{eq:defF}. Moreover, in the spirit of~\cite{bacomi-cv}, we show that those conditions may be weakened upon assuming that the approximated function is H\"older continuous. The theorem reads as follows.
\begin{theorem}[General approximation]\label{theo:genapprox}
        Let $\Omega \subseteq \rn$ be an open and bounded Lipschitz domain and let $\Uset$ be an open set such that $\Omega \Subset \Uset$. Recall the definitions of classes $\cZ^\vk$ and $A_r$ from~\eqref{eq:defZ} and~\eqref{eq:defMuck}. Let $\cF$ be defined as in~\eqref{eq:defF} for some numbers $p, q$ satisfying $1 \leq p \leq q < \infty$ and non-negative function $a$ satisfy $a = \bfun \wfun$, where $\wfun \in A_q(\Uset) \cap L^\infty(\Omega)$ and $\bfun : U \to [0, \infty)$. Let $u \in W^{1, 1}_0(\Omega)$ be such that $\cF[u;\Omega] < \infty$. Suppose that either
        \begin{enumerate}[i)]
            \item $\bfun \in \cZ^{\vk}(\Omega)$ for $\vk > 0$ such that $q \leq p + \vk\max(1, p/N)$ or
            \item $u \in C^{0, \gamma}(\Omega)$ and $\bfun \in \cZ^{\vk}(\Omega)$ for some $\gamma \in (0, 1)$ and $\vk > 0$ such that $q \leq p + \frac{\vk}{1-\gamma}$.
        \end{enumerate}
        Then there exists a sequence $\{u_n\}_n \subset C_c^{\infty}(\Omega)$ such that
    \begin{equation}\label{eq:theo2}
        \cF[u_n;\Omega] \xrightarrow{n \to \infty} \cF[u;\Omega] \quad \text{ and } \quad \|u - u_n\|_{W^{1, 1}(\Omega)} \xrightarrow{n \to \infty} 0\,.
    \end{equation}
\end{theorem}
The conclusion~\eqref{eq:theo2} is also sometimes referred to as the absence of Lavrentiev's gap. The result is more general than~\eqref{eq:LavAbs}, and reveals more connection to the theory of function spaces, namely, Musielak-Orlicz-Sobolev spaces of double phase type, briefly discussed in Section~\ref{sec:Musielak}.

For the sake of completeness, let us formulate Theorem~\ref{theo:genapprox} in terms of regular weights. The following corollary is a direct consequence of Theorems~\ref{theo:decomp} and~\ref{theo:genapprox}. 
\begin{corollary}\label{coro:genapp}
        Let $\Omega \subseteq \rn$ be an open and bounded Lipschitz domain and let $\Uset$ be an open set such that $\Omega \Subset \Uset$. Let $\cF$ be defined as in~\eqref{eq:defF} for some numbers $p, q$ satisfying $1 \leq p \leq q < \infty$, and non-negative function $a \in C^{k, \alpha}(U)$ for some numbers $k \in \N$, $\alpha \in (0, 1]$. Let $u \in W^{1, 1}_0(\Omega)$ be such that $\cF[u;\Omega] < \infty$. Suppose that either
        \begin{enumerate}[i)]
            \item $q \leq p + (k+\alpha)\max(1, p/N)$ or
            \item $u \in C^{0, \gamma}(\Omega)$ and $q \leq p + \frac{k+\alpha}{1-\gamma}$ for some $\gamma \in (0, 1)$.
        \end{enumerate}
        Then there exists a sequence $\{u_n\}_n \subset C_c^{\infty}(\Omega)$ such that
    \begin{equation*}
        \cF[u_n;\Omega] \xrightarrow{n \to \infty} \cF[u;\Omega] \quad \text{ and } \quad \|u - u_n\|_{W^{1, 1}(\Omega)} \xrightarrow{n \to \infty} 0\,.
    \end{equation*}
\end{corollary}
\subsection{Global approximation result}
The results of Theorems~\ref{theo:Lavdecomp} and~\ref{theo:genapprox} are essentially local, as we require $a$ to be defined on a bigger set than $\Omega$. Although it is not a setback for further developments, one may wonder when a global result can be obtained. Using known criteria for extending Muckenhoupt weights (\cite[p. 438]{Rubio}), we deduce the following global version of Theorem~\ref{theo:genapprox}. 
\begin{corollary}\label{coro:global}
        Let $\Omega \subseteq \rn$ be an open and bounded Lipschitz domain. Recall the definitions of classes $\cZ^\vk$ and $A_r(U)$ from~\eqref{eq:defZ} and~\eqref{eq:defMuck}. Let $\cF$ be defined as in~\eqref{eq:defF} for some numbers $p, q$ satisfying $1 \leq p \leq q < \infty$ and non-negative function $a$ satisfy $a = \bfun \wfun$, where $\bfun \in \SP^{\vk}(\Omega)$, $q \leq p + \vk\max(1, p/N)$ and $\wfun \in L^\infty(\Omega)$ satisfies
        \begin{equation*}
            \sup_{B \subseteq \rn} \left( \dashint_{B \cap \Omega} \wfun \d z \right) \left( \dashint_{B \cap \Omega} \wfun^{-\frac{1}{q-1}} \d z\right)^{q-1} < \infty\,,
        \end{equation*}
        where the supremum is taken over all balls in $\rn$. Then, for any function $u \in W^{1, 1}_0(\Omega)$ such that $\cF[u;\Omega] < \infty$, there exists a sequence $\{u_n\}_n \subset C_c^{\infty}(\Omega)$ such that
    \begin{equation*}
        \cF[u_n;\Omega] \xrightarrow{n \to \infty} \cF[u;\Omega] \quad \text{ and } \quad \|u - u_n\|_{W^{1, 1}(\Omega)} \xrightarrow{n \to \infty} 0\,.
    \end{equation*}
\end{corollary}
Under the same condition as in Corolary~\ref{coro:global}, one can obtain the absence of Lavretiev's phenomenon as in~\eqref{eq:LavAbs}. The reasoning would be precisely the same as the proof of Theorem~\ref{theo:Lavdecomp}, but applying Corollary~\ref{coro:global} instead of Theorem~\ref{theo:genapprox}.
\subsection{Density of smooth functions in Musielak-Orlicz-Sobolev spaces of double phase type}\label{sec:Musielak}
The density of smooth functions in Musielak-Orlicz-Sobolev spaces is closely tied to the study of Lavrentiev's gap. In fact, in the case of the double phase, the conditions for the former coincide with those for the latter. Let us thus define the Musielak–Orlicz–Sobolev space associated with~\eqref{eq:defF} and state the density result. For more information on such spaces, we refer to~\cite{C-b, hahab}.

For a given functional $\cF$, defined as in~\eqref{eq:defF}, and for any open set $\Uset \subset \rn$, we define the space
\begin{equation}\label{eq:defW}
    W(\Uset) \coloneqq \left\{ u \in W^{1, 1}_0(\Uset): \cF[u;\Uset] < \infty \right\}\,,
\end{equation}
equipped with the Luxembourg norm
\begin{equation*}
    \|u\|_{W(\Uset)} \coloneqq \inf \left\{s > 0 :  \int_{\Uset} \left( \left| \frac{\nabla u(x)}{s} \right|^p + a(x)\left| \frac{\nabla u(x)}{s} \right|^q \right)\d x \leq 1 \right\}\,.
\end{equation*}
As integrand of~\eqref{eq:defF} is doubling, one can show that for any sequence of functions $\{u_n\}_n \subset W(U)$ and a function $u \in W(U)$, it holds
\begin{equation*}
    \|u_n - u\|_{W(U)} \xrightarrow{n \to \infty} 0 \quad \iff \quad \cF[u_n - u] \xrightarrow{n \to \infty} 0\,.
\end{equation*}
Thus, we may infer the following corollary from Theorems~\ref{theo:decomp} and~\ref{theo:genapprox}.
\begin{corollary}\label{theo:dense}
    Let $\Omega \subseteq \rn$ be an open and bounded Lipschitz domain and let $\Uset \subseteq \rn$ be such that $\Omega \Subset \Uset$. Let $\cF$ be defined as in~\eqref{eq:defF} for $1 \leq p, q < \infty$, non-negative $a \in C^{k, \alpha}(U)$, $k \in \N$, $\alpha \in (0, 1]$, satisfying $q \leq p + (k+\alpha)\max(1, p/N)$. Then $C_c^{\infty}(\Omega)$ is dense in $W(\Omega)$, defined in~\eqref{eq:defW}, in norm topology.
\end{corollary}
One of the applications of Theorem~\ref{theo:dense} is that it can be used in the study of PDEs with operators of double phase type. We refer the reader to~\cite[Section 4.1.4]{C-b}, where it is explained how the density result as in~\ref{theo:dense} may be used to prove the existence of renormalized solutions.
\section{Proof of Theorem~\ref{theo:decomp}}\label{sec:proofs1}
We aim to prove that a regular weight can be decomposed into a product of weights, one from the $\mathcal{Z}$ class and the other from the Muckenhoupt class, with suitable parameters. The main ingredients of the proof are Lemma~\ref{lem:Zpart} and Proposition~\ref{prop:Muck}, each one dealing with one part of the decomposition. Lemma~\ref{lem:Zpart} says that the original weight may be forced to belong to the sufficient $\cZ$ class if one perturbs the weight with its derivatives. The proof of this fact essentially consists of Taylor's expansion and Young's inequality. 

In the proof of Proposition~\ref{prop:Muck}, we deal separately with three different types of balls, which we call small, medium-sized, and large. The small ones are essentially the ones on which the function identified in Lemma~\ref{lem:Zpart} may be treated as a constant one. We then show that on each small ball one can find a dominating derivative of the weight --- almost constant on the ball, in the sense of Lemma~\ref{lem:Muck}. This lemma says that if some derivative of the weight is almost constant on a given ball, then one can assert control over an integral of the weight in negative power. The proof involves a careful comparison of the weight to its Taylor polynomial, along with several technical details, particularly Lemma~\ref{lem:poly}.

The medium-sized balls share their characteristics with the small ones, but are also sufficiently large. Using similar techniques as before, we can prove that on those balls, solely the second term of~\eqref{eq:defMuck} may be estimated sufficiently so that the Muckenhoupt condition is satisfied. This observation enables us to cover the last type of balls — large ones — with medium-sized ones, thereby inferring the Muckenhoupt condition.  \newline 

Let us introduce some notation we shall use throughout the proof. For every symmetric multilinear operator $\Aform : (\rn)^{\ell} \to \R$, we denote by $\|\Aform\|$ its operator norm
\begin{equation*}
   \|\Aform \| \coloneqq \sup_{|v| \leq 1} \Aform[v^{\ell}]\,,
\end{equation*}
where $v^{\ell}$ denotes a tuple $(v, v, \dots, v) \in (\rn)^{\ell}$, where $v$ is taken $\ell$ times. Whenever $\ell = 0$, we shall identify $\Lambda$ with a real number, and we shall interpret $\Lambda[v^\ell]$ as $\Lambda$, for any $v \in \rn$. In particular, then $\|\Lambda\| = |\Lambda|$.

Recall that for any open set $\anyopen \subseteq \rn$, $k \in \N$ and $\alpha \in (0, 1]$, we say that function $a : \anyopen \to \R$ belongs to class $C^{k, \alpha}$ for some numbers, if $a$ has all the derivatives up to order $k$ and

\begin{equation*}
    [D^ka]_{C^{0, \alpha}(\anyopen)} = \sup_{x, y \in \anyopen} \frac{\|D^ka(x) - D^ka(y)\|}{|x-y|^{\alpha}} < \infty\,.
\end{equation*}

We shall now proceed with proving several facts needed to conclude Theorem~\ref{theo:decomp}. We start with the lemma which identifies the first factor of $a$ in Theorem~\ref{theo:decomp}.
\begin{lemma}\label{lem:Zpart}
    Let $a \in C^{k, \alpha}(\Uset)$ for an open and bounded set $\Uset$. We define function $\bfun$ as
        \begin{equation}\label{eq:defb}
        \bfun(x) \coloneqq \sum_{i=0}^{k} \|D^i a(x)\|^{\frac{k+\alpha}{k + \alpha - i}}\,.
    \end{equation}
    Then, for any open $\Omega$ such that $\Omega \Subset \Uset$, it holds $\bfun \in \SP^{k+\alpha}(\Omega)$.
\end{lemma}
\begin{proof}
    We have to prove that there exists a constant $C > 0$ such that for any $x, y \in \Omega$, we have
    \begin{equation}\label{eq:margoal1}
        \bfun(x) \leq C(\bfun(y) + |x-y|^{k+\alpha})\,.
    \end{equation}
    Let us denote $r \coloneqq \dist(\partial \Uset, \Omega)$ and observe that whenever $|x-y| \geq r$ and $C \geq \frac{\|\bfun\|_{L^{\infty}}}{r^{k+\alpha}}$, then~\eqref{eq:margoal1} holds. Therefore, we may focus on the case $|x - y| < r$. As for all $x \in \Omega$ we have that $B(x, r) \subset \Uset$, we can reformulate~\eqref{eq:margoal1} into the assertion that for all $x \in \Omega$ and $h \in B(0, r)$, it holds
    \begin{equation*}
        \bfun(x + h) \leq C(\bfun(x) + |h|^{k+\alpha})\,.
    \end{equation*}
    Let us thus fix $x \in \Omega$, $h \in B(0, r)$, and, moreover, $\ell \in \{0, \dots, k\}$. By Taylor expansion, there exists $\xi \in [x, x + h]$ such that
    \begin{align*}
        D^\ell a(x+h) &= \sum_{i=\ell }^{k-1} \frac{D^ia(x)[h^{i-\ell  }]}{(i-\ell  )!} + \frac{D^ka(\xi)[h^{k-\ell }]}{(k-\ell )!}\\
        &= \sum_{i=\ell }^{k} \frac{D^ia(x)[h^{i-\ell  }]}{(i-\ell  )!} + \frac{(D^ka(\xi) - D^ka(x))[h^{k-\ell }]}{(k-\ell )!}\,,
    \end{align*}
    which gives
    \begin{equation}\label{eq:mar1}
        \|D^\ell a(x+h)\| \leq \sum_{i=\ell }^{k} \|D^ia(x)\||h|^{i-\ell  } + \Ca |h|^{k-\ell+\alpha}\,,
    \end{equation}
    where $\Ca $ is H\"older constant of $D^ka \in C^{0, \alpha}$. Applying Young's inequality below the sum in the right-hand side of~\eqref{eq:mar1}, we get
    \begin{equation*}
        \|D^\ell a(x+h)\| \leq \sum_{i=\ell }^{k} \|D^ia(x)\|^{\frac{k+\alpha-\ell }{k+\alpha - i}} + (\Ca  + k)|h|^{k-\ell+\alpha}\,.
    \end{equation*}
    Applying power $\frac{k+\alpha}{k+\alpha-\ell }$ to both sides of the last display, we get for some $c = c(\Ca, k, \alpha) > 0$ that
    \begin{equation}\label{eq:mar2}
        \|D^\ell a(x+h)\|^{\frac{k+\alpha}{k+\alpha-\ell }} \leq c\left( \sum_{i=\ell }^{k} \|D^ia(x)\|^{\frac{k+\alpha}{k+\alpha - i}} + |h|^{k+\alpha} \right)\,.
    \end{equation}
    Summing~\eqref{eq:mar2} over all $\ell \in \{0, 1, \dots, k\}$, we get
    \begin{align*}
        \bfun(x+h) = \sum_{i=0}^{k} \|D^ia(x+h)\|^{\frac{k+\alpha}{k + \alpha - i}} &\leq (k+1)c\left( \sum_{i=0}^{k} \|D^ia(x)\|^{\frac{k+\alpha}{k+\alpha - i}} + |h|^{k+\alpha} \right)\\
        &= c(\bfun(x) + |h|^{k+\alpha})\,,
    \end{align*}
    which is the desired result.
\end{proof}
Lemma~\ref{lem:Zpart} clearly shows also what $\wfun$ in Theorem~\ref{theo:decomp} shall look like. To prove that such an $\wfun$ satisfies a suitable Muckenhoupt condition, we start with the following technical lemma.
\begin{lemma}\label{lem:poly}
Let $k \in \N$ and let $P$ be a polynomial of degree $k$, such that
\begin{equation*}
    P(t) = \sum_{i=0}^{k} \Ai_it^i\,, \text{ where $\Ai_k > 0\,.$}
\end{equation*}
Assume that $P$ is non-negative on an interval $[0, T]$ for some $T > 0$. Fix $\ve \in (0, 1]$. Then, there exists a constant $c_{k, \ve} > 0$, depending only on $k$ and $\ve$, such that there exists a finite family of pairwise disjoint intervals $\{[\aii_i, \bii_i]\}_{i=1}^{m}$ included in $[0, T]$, such that $\bii_i \leq c_{k, \ve}\aii_i$ for every $i = 1, \dots, m$, and
\begin{equation*}
    t \not\in \bigcup_{i=1}^{m} [\aii_i, \bii_i] \implies \sum_{i=0}^{k-1} \Ai_it^i + \ve \Ai_kt^k \geq \frac{\ve}{2^{k}}\Ai_kt^k\,.
\end{equation*}
\end{lemma}
\begin{proof}
    Given $P$, let us denote polynomials $P_1$ and $P_2$ via
    \begin{equation*}
        P_1(t) \coloneqq \sum_{i=0}^{k-1} \max(\Ai_i, 0)t^i\,, \quad P_2(t) \coloneqq \sum_{i=0}^{k-1} \max(-\Ai_i, 0)t^i\,,
    \end{equation*}
    so that
    \begin{equation*}
        \sum_{i=0}^{k-1} \Ai_it^i = P_1(t) - P_2(t) \quad \text{ and } \quad P(t) = P_1(t) - P_2(t) + \Ai_kt^k\,.
    \end{equation*}
    Proceeding by induction with respect to $k$, we shall prove a stronger result. Namely, that there exists a finite family of pairwise disjoint intervals $\{[\aii_i, \bii_i]\}_{i=1}^{m}$ included in $[0, T]$, such that $\bii_i \leq c_{k, \ve}\aii_i$ for $i = 1, \dots, m$, and
\begin{equation}\label{eq:goalpoly}
    t \not\in \bigcup_{i=1}^{m} [\aii_i, \bii_i] \implies P_1(t) - P_2(t) + \ve \Ai_kt^k \geq \frac{\ve}{2^{k}}\left( P_1(t) + \Ai_kt^k \right)\,.
\end{equation}
At first, observe that $0 \leq P(0) = \Ai_0$, which means that cases $k = 0$ and $k = 1$ are trivial. Suppose we have the result for $k - 1 \geq 1$, and proceed with proving the result for $k$. 

We may assume that $P_1$ is positive on $(0, +\infty)$. Indeed, if it is not, then it is identically equal to $0$, and so on $[0, T]$ we have $0 \leq P(t) = \Ai_kt^k - P_2(t)$, which may be true only if $P_2$ is identically equal $0$, as $P_2$ has only non-negative coefficients and is of degree lesser than $k$. Thus, we have $P(t) = \Ai_k t^k$, for which~\eqref{eq:goalpoly} is trivially satisfied. We shall thus assume from now on that $P_1$ is positive on $(0, +\infty)$.

We notice that the function $t \mapsto P_1(t) - \Ai_kt^k$ has exactly one root on $(0, +\infty)$. Indeed, the function is non-negative close to $0$, because of the positivity of $P_1$, and is negative for sufficiently large $t$, so there exists the smallest $t_0 > 0$ such that $P_1(t_0) = \Ai_kt_0^k$. As $P_1$ has non-negative coefficients and is of degree at most $k-1$, for all $c > 1$, we have
    \begin{equation*}
        P_1(ct) - \Ai_k(ct)^k \leq c^{k-1}P_1(t) - c^k\Ai_kt^k < c^k(P_1(t)-\Ai_kt^k) = 0\,.
    \end{equation*}
    Therefore, $t_0$ is the unique root of the function $t \mapsto P_1(t) - \Ai_kt^k$ on $(0, +\infty)$. In particular, for $t \geq t_0$ we have $P_1(t) \leq \Ai_kt^k$. For each $t \in [t_0, T]$, we then have
    \begin{equation*}
        0 \leq P(t) = P_1(t) + \Ai_kt^k - P_2(t) \leq 2\Ai_kt^k - P_2(t)\,,
    \end{equation*}
    i.e., $P_2(t) \leq 2\Ai_kt^k$. Note that $P_2$ is of degree lower than $k-1$, and has non-negative coefficients. Therefore
    \begin{equation*}
        P_2(4\ve^{-1}t) \leq 4^{k-1}\ve^{1-k}P_2(t) \leq 2 \cdot 4^{k-1}\ve^{1-k}\Ai_kt^k = \frac{\ve}{2}\Ai_k(4\ve^{-1}t)^k\,.
    \end{equation*}
    Thus,
    \begin{equation*}
        P_1(4\ve^{-1}t) - P_2(4\ve^{-1}t) + \ve \Ai_k(4\ve^{-1}t)^k \geq P_1(4\ve^{-1}t) + \frac{\ve}{2}\Ai_k(4\ve^{-1}t)^k \geq \frac{\ve}{2^k}\left(P_1(4\ve^{-1}t) + \Ai_k(4\ve^{-1}t)^k \right)\,.
    \end{equation*}
    Equivalently
    \begin{equation}\label{eq:0807-8}
        P_1(t) - P_2(t) + \ve \Ai_kt^k \geq \frac{\ve}{2^k}\left(P_1(t) + \Ai_kt^k \right)\, \quad \text{for $t \in [4\ve^{-1}t_0, T]$.}
    \end{equation}
    Let us remark that we interpret interval $[s, \tau]$ as $\emptyset$ whenever $s > \tau$.
    
    Now, let us take $t_1$ to be equal to $t_0$ if $t_0 \in [0, T]$, and to $T$ otherwise. Recall that $P_1(t) \geq \Ai_kt^k$ for $t \leq t_0$. We now focus on $t \leq t_1$. Then, $0 \leq P(t) \leq 2P_1(t) - P_2(t)$, which implies $P_2(t) \leq 2P_1(t)$. Let $\ell$ be the degree of polynomial $P_1$, so that $\Ai_l \geq 0$ and $\Ai_{\ell+1}, \dots \Ai_{k-1} < 0$. Note that
    \begin{equation}\label{eq:0608-1}
        - \sum_{i=\ell+1}^{k-1} \Ai_it^i \leq P_2(t) \leq 2P_1(t)\,.
    \end{equation}
    On the right-hand side, we have a polynomial of order $\ell$, with non-negative coefficients, while on the left-hand side, a polynomial with non-negative coefficients and with all coefficients up to order $\ell$ equal $0$. This in conjunction with~\eqref{eq:0608-1} means that for every $c \leq 1$, we have
    \begin{equation*}
        -\sum_{i=\ell+1}^{k-1} \Ai_i(ct)^{i} \leq c^{\ell+1} \left(- \sum_{i=\ell+1}^{k-1} \Ai_it^i \right) \leq 2c^{\ell+1}P_1(t) \leq 2cP_1(ct)\,.
    \end{equation*}
    If we take $c = \frac{\ve}{2^{k+1}}$, we obtain that on interval $[0, \tfrac{\ve}{2^{k+1}}t_1]$, it holds
    \begin{equation}\label{eq:0807-7}
        -\sum_{i=\ell+1}^{k-1} \Ai_it^{i} \leq \frac{\ve}{2^{k}}P_1(t)\,.
    \end{equation}
    We now apply the inductive hypothesis to the polynomial $\sum_{i=0}^{\ell} \Ai_it^i$ on interval $[0, \tfrac{\ve}{2^{k+1}}t_1]$. We get family $\{[\aii_i, \bii_i]\}_{i=1}^{m}$ of pairwise disjoint intervals in $[0, \tfrac{\ve}{2^{k+1}}t_1]$, satisfying $\bii_i \leq c_{k-1, \ve}\aii_i$ and such that
    \begin{equation*}
        t \not\in \bigcup_{i=1}^{m} [\aii_i, \bii_i] \implies \sum_{i=0}^{\ell-1} \Ai_it^i + \ve \Ai_lt^{\ell} \geq \frac{\ve}{2^{\ell}}P_1(t) \geq \frac{\ve}{2^{k-1}}P_1(t)\,.
    \end{equation*}
    This in conjunction with~\eqref{eq:0807-7} means that whenever $t \not\in \bigcup_{i=1}^{m} [\aii_i, \bii_i]$, then
    \begin{align}\label{eq:0807-9}
        P_1(t) - P_2(t) + \ve \Ai_kt^k &= \sum_{i=0}^{\ell} \Ai_it^i - \sum_{i=\ell+1}^{k-1} \Ai_it^i + \ve \Ai_kt^k \nonumber\\
        &\geq \frac{\ve}{2^{k-1}}P_1(t) - \frac{\ve}{2^{k}}P_1(t) + \ve \Ai_kt^k \geq \frac{\ve}{2^{k}}\left( P_1(t) + \Ai_kt^k  \right)\,.
    \end{align}
    Together,~\eqref{eq:0807-8} and~\eqref{eq:0807-9} imply that
    \begin{equation*}
        t \not\in \bigcup_{i=1}^{m} [\aii_i, \bii_i] \cup \left[ \frac{\ve}{2^{k+1}}t_1, \min(4\ve^{-1}t_1, T) \right] \implies P_1(t) - P_2(t) + \ve \Ai_kt^k \geq \frac{\ve}{2^{k}}\left(P_1(t) + \Ai_kt^k \right)
    \end{equation*}
    We define a new family of intervals as follows. If $\bii_i \neq \frac{\ve}{2^{k+1}}t_1$ for all $i \in \{1, \dots, m\}$, then we simply take $\{[\aii_i, \bii_i]\}_{i=1}^{m}$ together with interval $\left[ \frac{\ve}{2^{k+1}}t_1, \min(4\ve^{-1}t_1, T) \right]$. We then obtain the conclusion as long as $c_{k, \ve} \geq \max(c_{\ell, \ve}, \ve^{-2}2^{k+3})$. On the other hand, if there exists $i$ such that $\bii_i = \frac{\ve}{2^{k+1}}t_1$, then we take family
    \begin{equation*}
        \{[\aii_j, \bii_j]\}_{j=1}^{i-1} \cup \big\{\left[ \aii_i, \min(4\ve^{-1}t_1, T) \right]\big\} \cup \{[\aii_j, \bii_j]\}_{j=i+1}^{m}\,.
    \end{equation*}
    Then, we get the conclusion as long as $c_{k, \ve} \geq c_{\ell, \ve}\ve^{-2}2^{k+3}$, since
    \begin{equation*}
        \frac{\min(4\ve^{-1}t_1, T)}{\aii_i}  \leq \frac{\bii_i}{\aii_i} \cdot \frac{4\ve^{-1}t_1}{\bii_i}  \leq c_{\ell, \ve}\ve^{-2}2^{k+3}\,.
    \end{equation*}
    In any case, we get the result with $c_{k, \ve} = \max_{i \leq k-1} c_{i, \ve}\ve^{-2}2^{k+3}$.
\end{proof}
We also prove the following minor lemma.
\begin{lemma}\label{lem:minima}
    Let $I \subseteq \R$ be an open interval and let the function $g : I \to \R$ be $\ell$ times differentiable, $\ell \geq 1$, and let $g^{(\ell)}$ be strictly positive. Then, the function $g$ has at most $\ell-1$ local minima.
\end{lemma}
\begin{proof}
    We proceed by induction with respect to $\ell$. Clearly, in the case of $\ell = 1$, the function $g$ is strictly increasing and thus has no local minima.

    Assume that the statement holds for some $\ell \in \N$, $\ell \geq 1$. Take function $g : I \to \R$ differentiable $\ell+ 1$ times and such that $g^{(\ell+1)}$ is strictly positive. The inductive assumption implies that $g'$ has at most $\ell-1$ local minima. This means that $g'$ can change its sign at most $\ell$ times, and thus, $g$ itself can have at most $\ell$ local minima.
\end{proof}
The following lemma says that if some derivative of the weight is almost constant on some ball, then an integral resembling the one appearing in the Muckenhoupt condition on this ball can be controlled.
\begin{lemma}\label{lem:Muck}
    Let $B \subseteq \rn$ be a ball of radius $r > 0$, let $\ell \in \N$, $\ve \in (0, 1]$ and let $a \in C^{\ell}(\sqrt{N}B)$ be non-negative function. Assume further that there exists a vector $v \in \rn$, $|v| = 1$, and $\Mbound> 0$ such that for all $z \in \sqrt{N}B$ it holds
    \begin{equation}\label{eq:Kbounds}
        \ve \Mbound \leq D^\ell a(z)[v^{\ell}] \leq \Mbound\,.
    \end{equation}
    Then, for every $\betapow > \ell$, there exists a constant $c = c(\beta, N, \ve)$ such that
    \begin{equation*}
        \left( \dashint_{B} a^{-\frac{1}{\betapow}}\d z \right)^{\betapow} \leq c\Mbound^{-1}r^{-\ell}\,.
    \end{equation*}
\end{lemma}
\begin{proof} The result is trivial when $\ell = 0$. Let us thus assume that $\ell \geq 1$.
We take a cube $Q$ with the same center as $B$, sidelength of $2r$, and all its sides either orthogonal or parallel to $v$. Note that $B \subseteq Q \subseteq \sqrt{N}B$, which gives
    \begin{equation}\label{eq:0807-10}
        \left( \dashint_{B} a^{-\frac{1}{\betapow}} \d z \right)^{\betapow} \leq N^{\frac{1}{2}\betapow}\left( \dashint_{Q} a^{-\frac{1}{\betapow}} \d z\right)^{\betapow}\,.
    \end{equation}
    Let $W$ be a set of orthogonal vectors to $v$, with all coordinates lesser than $r$. For any vector $w \in W$, let us consider an interval $V_w = [x + w - rv, x + w + rv]$. Note that
    \begin{equation}\label{eq:0808-6}
        \dashint_{Q} a^{-\frac{1}{\betapow}} \d z = \dashint_{W} \left(\dashint_{V_w} a^{-\frac{1}{\betapow}} \d s\right)\,\d w\,.
    \end{equation}
    Let us fix $w \in W$ and consider a function
    \begin{equation*}
        a_w(t) \coloneqq a(w+x+tv) \text{ for } |t| \leq r\,.
    \end{equation*}
    Observe that $a_w^{(\ell)}(t) = D^\ell a(w+x+tv)[v^{\ell}]$, and thus, by an assumption, we have
    \begin{equation*}
        \ve \Mbound\leq a_w^{(\ell)}(t) \leq \Mbound\,.
    \end{equation*}
    In particular, $a_w^{(\ell)}$ is strictly positive on interval $[-r, r]$. By Lemma~\ref{lem:minima}, $a_w$ restricted to $[-r, r]$ has at most $\ell+1$ local minima, possibly including the endpoints of $[-r, r]$. Let us denote those minima by $t_1, t_2, \dots, t_m$, $m \leq l + 1$, and let us also take numbers $\ai_1, \dots, \ai_{m+1}$ such that
    \begin{equation*}
        -r = \ai_1 \leq t_1 \leq \ai_2 \leq t_2 \leq \dots \leq t_m \leq \ai_{m+1} = r
    \end{equation*}
    and $a_w$ is decreasing on $[\ai_i, t_i]$ and increasing on $[t_i, \ai_{i+1}]$, for every $i = 1, \dots, m$. We fix such an $i$. Taking any $h$ such that $h \leq \ai_{i+1} - t_i$, using Taylor expansion and~\eqref{eq:Kbounds} we observe that
    \begin{equation}\label{eq:0804-1}
        \sum_{j=0}^{\ell-1} \frac{a_w^{(j)}(t_i)h^j}{j!} + \frac{\ve\Mbound h^{\ell}}{\ell !} \leq a_w(t_i + h) \leq \sum_{j=0}^{\ell-1} \frac{a_w^{(j)}(t_i)h^j}{j!} + \frac{\Mbound h^{\ell}}{\ell !}\,.
    \end{equation}
    As $a_w(t_i) \leq a_w(t_i + h)$, the second inequality of~\eqref{eq:0804-1} implies
    \begin{equation}\label{eq:2407}
        0 \leq \sum_{j=1}^{\ell-1} \frac{a_w(t_i)h^j}{j!} + \frac{\Mbound h^{\ell}}{\ell !}\,.
    \end{equation}
    Lemma~\ref{lem:poly} may be thus applied to the right-hand side of~\eqref{eq:2407} on interval $[0, s_{i+1} - t_i]$. In turn, we get that there exists a finite family of pairwise disjoint intervals $[\tilde \aii_i, \tilde \bii_i]_{i=1}^{m_2}$ such that $\bii_i \leq c_{\ell , \ve}\aii_i$ and
    \begin{equation}\label{eq:0807-3}
        h \not\in \bigcup_{i=1}^{m_2} [\tilde \aii_i, \tilde \bii_i] \quad \implies \quad  a_w(t_i + h) \geq \frac{\ve \Mbound h^{\ell}}{2^{k}\ell!}\,,
    \end{equation}
    where we also used the first inequality of~\eqref{eq:0804-1}. By monotonicity of $a_w$, if $a_w(t_i + h_0) \geq Ch_0^\ell$, then for $h \in [h_0, \min \left(c_{\ell , \ve}h_0 , \ai_{i+1} - t_i\right)]$, it holds 
    
    \begin{equation*}
    a_w(t_i + h) \geq a_w(t_i + h_0) \geq Ch_0^\ell \geq \frac{C}{c_{\ell , \ve}^{\ell}}h^{\ell}\,. 
    \end{equation*}
    
    This together with~\eqref{eq:0807-3} means that for some $c = c(\beta, \ve) > 0$, it holds
    \begin{equation*}
        a_w(t_i + h) \geq \frac{\Mbound h^{\ell}}{c} \text{ for } h \leq \ai_{i+1} - t_i\,.
    \end{equation*}
    In particular,
    \begin{equation}\label{eq:0807-4}
        \int_{t_i}^{\ai_{i+1}} a_w^{-\frac{1}{\betapow}} \d s \leq cK^{-\frac{1}{\betapow}}\int_{t_i}^{\ai_{i+1}} h^{-\frac{l}{\beta}}\d h \leq  c\Mbound^{-\frac{1}{\betapow}}(\ai_{i+1}-t_i)^{1 - \frac{\ell}{\betapow}} \leq c\Mbound^{-\frac{1}{\betapow}}r^{1 - \frac{\ell}{\betapow}}\,.
    \end{equation}
    With a similar reasoning, one can also establish that
    \begin{equation}\label{eq:0807-5}
        \int_{\ai_i}^{t_i} a_w^{-\frac{1}{\betapow}} \d s  \leq c\Mbound^{-\frac{1}{\betapow}}r^{1 - \frac{\ell}{\betapow}}\,.
    \end{equation}
    Summing over all $i$ inequalities~\eqref{eq:0807-4} and~\eqref{eq:0807-5}, we get
    \begin{equation*}
        \dashint_{V_w} a_w^{-\frac{1}{\betapow}} \d s \leq 2mc\Mbound^{-\frac{1}{\betapow}}r^{-\frac{\ell}{\betapow}} \leq 2(\ell+1)c\Mbound^{-\frac{1}{\betapow}}r^{-\frac{\ell}{\betapow}}\,,
    \end{equation*}
    which, by~\eqref{eq:0808-6}, gives
    \begin{equation*}
        \left(\dashint_{Q} a^{-\frac{1}{\betapow}}\d z\right)^{\betapow} \leq c\Mbound^{-1}r^{-\ell}\,,
    \end{equation*}
    Applying~\eqref{eq:0807-10}, we get the desired result.    
\end{proof}
We are now in a position to prove the most crucial proposition for the proof of Theorem~\ref{theo:decomp}.
\begin{proposition}\label{prop:Muck}
    Let $a \in C^{k, \alpha}(\Uset)$ be non-negative function, where $\Uset$ is open and bounded subset of $\rn$, $k \in \N$, $\alpha \in (0, 1]$. We define function $\wfun : \Uset \to [0, 1]$ as
\begin{equation}\label{eq:defw}
    \wfun(x) = \begin{cases}
        \frac{a(x)}{\sum_{i=0}^{k} \|D^ia(x)\|^{\frac{k+\alpha}{k+\alpha - i}}}\, &\text{if }\|D^ia(x)\| \neq 0 \text{ for some $i$;}\\
        1\, &\text{otherwise.}
    \end{cases}
\end{equation}
    Then, for any open $\Omega$ such that $\Omega \Subset \Uset$, it holds $\wfun \in A_{k+\alpha+1}(\Omega)$.
\end{proposition}
\begin{proof}
    We define function $\bfun$ as in~\eqref{eq:defb}, so that $a = \bfun \wfun$. Let $\Omprim$ be any open set such that $\Omega \Subset \Omprim \Subset \Uset$. By Lemma~\ref{lem:Zpart}, we have that $\bfun \in \SP^{k+\alpha}(\Omprim)$, i.e., there exists a constant $\Cb \geq 1$ such that for all $x, y \in \Omprim$ it holds
    \begin{equation}\label{eq:defCb}
        \bfun(x) \leq \Cb  \left( \bfun(y) + |x-y|^{k+\alpha} \right)\,.
    \end{equation}
    Without loss of generality, we may assume that $[D^ka]_{C^{0, \alpha}(\Uset)} \leq \Cb $.
    Let us now define a sequence $\{\Cnum_n\}_n$ recursively via
    \begin{equation}\label{eq:defC}
    \Cnum_0 = \sqrt{N}\,, \quad \Cnum_{n+1} = 4\cdot 3^{n+1}  ((n+1)!) \cdot \left(\prod_{i=0}^{n} \Cnum_i \right) + 1\,. 
\end{equation}
    We also define numbers $\{\dnum_n\}_n$ by the formula
    \begin{equation*}
        \dnum_0 = \frac{1}{2}\,, \quad \dnum_{n+1} = \frac{1}{2}\min_{i \leq n} \frac{\alpha \dnum_i}{n+1+\alpha - i}\,.
    \end{equation*}
    Let us now define constant $\tildec = \tildec(\Omega, \Omprim, \|\bfun\|_{L^\infty(\Omega)}, k, \alpha, \Cb)$ by the following formula
    \begin{align*}
        \tildec \coloneqq \min \Biggl\{&\frac{\dist(\Omega, \partial \Omega')}{\Cnum_k \left(\|\bfun\|_{L^\infty(\Omega)}^{\frac{1}{k+\alpha}} + 1\right)}, \left(\frac{1}{2\Cb }\right)^{\frac{1}{k+\alpha}}, \frac{1}{2\Cnum_k(k+1)(2 \cdot 3^k\Cb )^{\frac{1}{\alpha}}}, \left(\frac{1}{4\Cb (k+1) (4\Cb k!)^{\frac{k+\alpha}{\alpha}}}\right)^{\frac{1}{\alpha\dnum_k}},\\
        &\left(\frac{\dist(\Omega, \partial \Omprim)}{(\Cnum_k + 2k!)\left(\|\bfun\|_{L^\infty(\Omega)}^{\frac{1}{k+\alpha}}+1\right)}\right)^{\frac{1}{\dnum_k}}, 2^{-\frac{1}{(1-\alpha)\dnum_k}}, \left(\frac{1}{\sqrt{N}}\right)^{\frac{1}{1-d_k}}\Biggr\}\,.
    \end{align*}
    For any $x \in \Omega$, we define
    \begin{equation}\label{eq:defR}
        R_x \coloneqq \tildec \bfun(x)^{\frac{1}{k+\alpha}}\,.
    \end{equation}
    From the choice of $\tildec$, we have that $B(x, \Cnum_kR_x) \subseteq \Omprim$ for any $x \in \Omega$, where $\Cnum_k$ is defined in~\eqref{eq:defC}, while $k$ comes from the statement. \newline
    
    {\bf Step 1: Muckenhoupt condition on small balls.} Let us take any ball $B = B(x, r)$, where $x \in \Omega$ and assume that there exists $x_0 \in \Omega$ such that $B \subseteq B(x_0, R_{x_0})$. In particular, $\bfun(x_0) > 0$ and $r \leq R_{x_0}$. Note that for any $y \in B\left(x_0, \left(\bfun(x_0)/(2\Cb) \right)^{\frac{1}{k+\alpha}}\right) \cap \Omprim$, we have, due to~\eqref{eq:defb}, that
    \begin{equation}\label{eq:0707-1}
        \frac{\bfun(x_0)}{2\Cb } \leq \bfun(y) \leq 2\Cb \bfun(x_0)\,.
    \end{equation}
    In particular,~\eqref{eq:0707-1} holds for every $y \in B$, due to $B \subseteq B(x_0, R_{x_0})$ and the choice of $\tildec$. Therefore,
    \begin{equation}\label{eq:0707-2}
        \left(\dashint_{B} \wfun \d z\right) \left( \dashint_{B} \wfun^{-\frac{1}{k+\alpha}}\d z \right)^{k+\alpha} = \left(\dashint_{B} \frac{a}{\bfun} \d z\right) \left( \dashint_{B} \left(\frac{\bfun}{a} \right)^{\frac{1}{k+\alpha}} \d z \right)^{k+\alpha} \leq 4 \Cb^2 \left( \dashint_{B} a \d z \right) \left( \dashint_{B} a^{-\frac{1}{k+\alpha}} \d z \right)^{k+\alpha}\,.
    \end{equation}
    Thus, it is sufficient to establish the Muckenhoupt condition on the ball $B$ only for the function $a$. Note that~\eqref{eq:0707-1} also implies that $\bfun(x) > 0$.

    We now take the smallest $\ell \in \{0, 1, \dots, k - 1\}$ such that
    \begin{equation}\label{eq:def-l}
        \|D^\ell  a(x)\| \geq 2\Cnum_{\ell}r \sup_{z \in \Cnum_{\ell+1}B} \|D^{\ell+1}a(z)\|\,,
    \end{equation}
    or, if there is no such number, we take $\ell = k$. Note that by the choice of $\tildec$, we have that $\Cnum_i B \subseteq B(x_0, \Cnum_i R_{x_0})\subseteq \Omprim$ for every $i \in \{0, 1, \dots, k\}$. We denote
    \begin{equation*}
        \Mbound\coloneqq \sup_{z \in \Cnum_{\ell}B} \|D^\ell  a(z)\|\,.
    \end{equation*}
    Note that $\Mbound\neq 0$. Indeed, if $\Mbound= 0$, then $\ell$-th derivative of $a$ vanishes on $B$, and thus also derivatives of higher order vanish on $B$. On the other hand, by the definition of $\ell$, we have that $D^i a(x) = 0$ for every $i \in \{0, 1, \dots, \ell-1\}$. Thus, we have that $\bfun(x) = 0$, but this is a contradiction. Therefore, $\Mbound> 0$.

    The definition of $\ell$ immediately gives that $\|D^{\ell-1}a(x)\| \leq 2\Cnum_{\ell-1}r \Mbound$, and further
    \begin{equation*}
        \sup_{z \in \Cnum_{\ell-1}B} \|D^{\ell-1}a(z)\| \leq \sup_{z \in \Cnum_{\ell-1}B} \|D^{\ell-1}a(z) - D^{\ell-1}a(x)\| + \|D^{\ell-1}a(x)\| \leq \Cnum_{\ell-1}r\sup_{z \in \Cnum_{\ell-1}B}\|D^\ell a(z)\| + 2\Cnum_{\ell-1}r\Mbound\leq 3\Cnum_{\ell-1}r\Mbound\,.
    \end{equation*}
    By an easy induction, we have
    \begin{align}
        \|D^{\ell-i}a(x)\| &\leq 2 \cdot 3^{i-1} \left( \prod_{j=\ell-i}^{\ell-1} \Cnum_j \right)r^i\Mbound\quad \text{ and } \label{eq:0807-1}\\
        \sup_{z \in C_{\ell-i}B}\|D^{\ell-i}a(z)\| &\leq 3^{i} \left( \prod_{j=\ell-i}^{\ell-1} \Cnum_j \right)r^i\Mbound \quad \text{ for $i = 0, \dots, \ell$.} \label{eq:2107-1}
    \end{align}
    Let us remark that we interpret an empty product as equal to $1$.
    
    We shall now prove that for any $y \in \Cnum_{\ell} B$, it holds
    \begin{equation}\label{eq:0707-3}
        \|D^\ell a(x) - D^\ell a(y)\| \leq \frac{1}{2}\|D^\ell a(x)\| \quad \text{ and } \quad \|D^\ell a(x)\| \geq \frac{1}{2}\Mbound\,.
    \end{equation}
    If $\ell < k$, then we simply use~\eqref{eq:def-l} to conclude that
    \begin{equation}\label{eq:0707-4}
        \|D^\ell a(x) - D^\ell a(y)\| \leq \Cnum_{\ell}r \sup_{z \in \Cnum_{\ell} B} \|D^{\ell+1}a(z)\| \leq \frac{1}{2}\|D^\ell a(x)\|\,.
    \end{equation}
    From~\eqref{eq:0707-4} and the definition of $\Mbound$, it also follows that
    \begin{equation*}
        \Mbound\leq \sup_{z \in \Cnum_{\ell} B} \|D^\ell a(z) - D^\ell a(x)\| + \|D^\ell a(x)\| \leq \frac{3}{2}\|D^\ell a(x)\|\,.
    \end{equation*}
    Inequalities~\eqref{eq:0707-3} thus follow if $\ell < k$. Let us now assume that $\ell = k$. Then,~\eqref{eq:0707-1} would follow immediately from $D^ka \in C^{0, \alpha}(\Uset)$, if we establish that $\|D^ka(x)\| \geq 2\Cb (\Cnum_kr)^{\alpha}$. Suppose that $\|D^ka(x)\| < 2\Cb (\Cnum_kr)^{\alpha}$. Then, using that $D^{k}a \in C^{0, \alpha}(\Uset)$, we get $\Mbound\leq 3\Cb (\Cnum_kr)^{\alpha}$, and from~\eqref{eq:0807-1}, we have for every $i \in \{0, 1, \dots, k\}$ that
    \begin{equation*}
        \|D^{k-i}a(x)\| < 2 \cdot 3^i \left( \prod_{j=k-i}^{k-1} \Cnum_j \right)r^{i+\alpha}\Cb \Cnum_k^{\alpha} \leq 2\cdot 3^i\Cb (\Cnum_kr)^{i+\alpha}
    \end{equation*}
    As $B \subseteq B(x_0, R_{x_0})$, and we have~\eqref{eq:0707-1}, we have
    \begin{equation*}
        r \leq R_{x_0} = \tildec \bfun(x_0)^{\frac{1}{k+\alpha}} \leq \tildec \left(2\Cb \bfun(x)\right)^{\frac{1}{k+\alpha}}\,.  
    \end{equation*}
    Thus, using the choice of $\tildec$, we get
    \begin{equation*}
        \bfun(x) = \sum_{i=0}^{k} \|D^{k-i}a(x)\|^{\frac{k+\alpha}{\alpha + i}} < \sum_{i=0}^{k} \left( 2 \cdot 3^i\Cb  \right)^{\frac{k+\alpha}{\alpha+i}}(\Cnum_kr)^{k+\alpha} \leq (k+1)\left(2 \cdot 3^k\Cb  \right)^{\frac{k+\alpha}{\alpha}} \Cnum_k^{k+\alpha} \tildec^{k+\alpha}2\Cb \bfun(x) \leq \frac{\bfun(x)}{2}\,,
    \end{equation*}
    which is clearly a contradiction. Therefore, $\|D^ka(x)\| \geq 2\Cb (\Cnum_kr)^{\alpha}$, and~\eqref{eq:0707-3} follows.

    Take a vector $v \in \rn$ such that $|v| = 1$ and $|D^\ell a(x)[v^{\ell}]| = \|D^\ell a(x)\|$. We shall prove that $D^\ell a(x)[v^{\ell}] = \|D^\ell a(x)\|$. Suppose that $D^\ell a(x)[v^{\ell}] = -\|D^\ell a(x)\|$. Then, there exists $\xi \in [x, x + \Cnum_{\ell}rv]$ such that
    \begin{align}\label{eq:0807-2}
        a(x + \Cnum_{\ell}rv) &= \sum_{i=1}^{\ell} \frac{(\Cnum_{\ell}r)^{\ell-i}(D^{\ell-i}a(x)[v^{\ell-i}])}{(\ell-i)!} + \frac{(\Cnum_{\ell}r)^{\ell}(D^\ell a(\xi)[v^{\ell}])}{\ell !} \nonumber\\
        &= \sum_{i=1}^{\ell} \frac{(\Cnum_{\ell}r)^{\ell-i}(D^{\ell-i}a(x)[v^{\ell-i}])}{(\ell-i)!} - \frac{(\Cnum_{\ell}r)^{\ell}\|D^\ell a(x)\|}{\ell !} + \frac{(\Cnum_{\ell}r)^{\ell}((D^\ell a(\xi)-D^\ell a(x))[v^{\ell}])}{\ell !}\,.
    \end{align}
    Using that $a$ is non-negative,~\eqref{eq:0707-3},~\eqref{eq:0807-1} and~\eqref{eq:0807-2}, we get
    \begin{align*}
        0 \leq a(x + \Cnum_{\ell}rv) &\leq \ell \Cnum_{\ell}^{\ell-1} 3^{\ell}r^{\ell}\Mbound\left(\prod_{j=0}^{\ell-1} \Cnum_j\right) - \frac{(\Cnum_{\ell}r)^{\ell}\|D^\ell a(x)\|}{\ell !} + \frac{(\Cnum_{\ell}r)^{\ell}\|D^\ell a(x)\|}{2\ell !}\\
        &\leq \ell \Cnum_{\ell}^{\ell-1} 3^{\ell}r^{\ell}M\left(\prod_{j=0}^{\ell-1} \Cnum_j\right) - \frac{(\Cnum_{\ell}r)^{\ell}\Mbound}{4\ell !} = \Mbound \Cnum_{\ell}^{\ell-1}r^{\ell} \left(\ell 3^{\ell}\left(\prod_{j=0}^{\ell-1} \Cnum_j\right) - \frac{\Cnum_{\ell}}{4\ell !} \right) < 0\,,
    \end{align*}
    where the last inequality is a consequence of the choice of $\Cnum_{\ell}$. We have a clear contradiction, therefore, $|D^\ell a(x)[v^{\ell}]| = \|D^\ell a(x)\|$.

    By~\eqref{eq:0707-3} and the fact that $\Cnum_0 \geq \sqrt{N}$, for all $z \in \sqrt{N}B$ it holds
    \begin{equation*}
        \frac{1}{4}\Mbound\leq D^\ell a(z)[v^{\ell}] \leq \Mbound\,.
    \end{equation*}
    
    Applying Lemma~\ref{lem:Muck} and~\eqref{eq:2107-1} with $i=\ell $, we finally obtain
    \begin{equation}\label{eq:step1}
        \left( \dashint_B a \d z \right)\left( \dashint_{B} a^{-\frac{1}{k+\alpha}} \d z \right)^{k+\alpha} \leq 3^{\ell} \left( \prod_{j=0}^{\ell-1} \Cnum_j \right)r^\ell \Mbound \cdot c(k, N, \alpha)K^{-1}r^{-\ell} = c(k, N, \alpha)\,.
    \end{equation}
    {\bf Step 2: More control over medium-sized balls.} Let $B = B(x, R_x)$ for some $x \in \Omega$, where $R_x$ is defined as in~\eqref{eq:defR}. Let us take the smallest $\ell \in \{0, 1, \dots, k\}$ such that
    \begin{equation}\label{eq:0907-1}
        \|D^\ell a(x)\| \geq 4\Cb \ell !\tildec^{\alpha \dnum_{\ell}}\bfun(x)^{\frac{k+\alpha-\ell }{k+\alpha}}\,.
    \end{equation}
    Note that such $\ell$ has to exist. Indeed, suppose that~\eqref{eq:0907-1} fails for all $\ell$. Then
    \begin{equation*}
        \bfun(x) = \sum_{i=0}^{k} \|D^ia(x)\|^{\frac{k+\alpha}{k+\alpha-i}} < \sum_{i=0}^{k} \left(4\Cb i!\tildec^{\alpha \dnum_i}\right)^{\frac{k+\alpha}{k+\alpha-i}} \bfun(x) \leq (k+1)\left(4k! \Cb \right)^{\frac{k+\alpha}{\alpha}}\tildec^{\alpha \dnum_k} \bfun(x) \leq \frac{\bfun(x)}{2}\,,
    \end{equation*}
    where the last inequality comes from the choice of $\tildec$. We obtained a clear contradiction, so the chosen $\ell$ exists.
    
    Observe that, due to the choice of $\tildec$, we have $B(x, 2\ell !\tildec^{\dnum_{\ell}}\bfun(x)^{\frac{1}{k+\alpha}}) \subseteq \Omprim$. We shall now prove that for every $y \in B(x, 2\ell !\tildec^{\dnum_{\ell}}\bfun(x)^{\frac{1}{k+\alpha}})$, it holds
    \begin{equation}\label{eq:0907-2}
        \|D^\ell a(x) - D^\ell a(y)\| \leq \frac{1}{2}\|D^\ell a(x)\|\,.
    \end{equation}
    If $\ell = k$, we use the fact that $D^ka \in C^{0, \alpha}(\Uset)$ and~\eqref{eq:0907-1} to estimate
    \begin{equation*}
        \|D^ka(x) - D^ka(y)\| \leq \Cb \left(2k! \tildec^{\dnum_k}\bfun(x)^{\frac{1}{k+\alpha}} \right)^{\alpha} \leq \frac{1}{2}(k!)^{\alpha - 1}\|D^ka(x)\| \leq \frac{1}{2}\|D^ka(x)\|
    \end{equation*}
    for all $y \in B(x, 2\ell !\tildec^{\dnum_{\ell}}\bfun(x)^{\frac{1}{k+\alpha}})$. Therefore, we get~\eqref{eq:0907-2} for $\ell = k$. If $\ell < k$, let us observe that, due to the fact that $\bfun \in \SP^{k+\alpha}(\Omprim)$, we have
    \begin{equation}\label{eq:0907-4}
        \sup_{y \in B(x, 2\ell !\tildec^{\dnum_{\ell}}\bfun(x)^{\frac{1}{k+\alpha}})}\bfun(y) \leq \Cb \left( \bfun(x) + (2\ell !)^{k+\alpha}\tildec^{(k+\alpha)\dnum_{\ell}}\bfun(x) \right) \leq \Cb \bfun(x)\left(1 + (2k!)^{k+\alpha}\tildec^{(k+\alpha)\dnum_k} \right) \leq 2\Cb \bfun(x)\,,
    \end{equation}
    where in the last inequality we used the choice of $\tildec$. Using~\eqref{eq:0907-4} in conjunction with the definition of $\bfun$, we get
    \begin{equation}\label{eq:0907-6}
        \sup_{y \in B(x, 2\ell !\tildec^{\dnum_{\ell}}\bfun(x)^{\frac{1}{k+\alpha}})} \|D^{i}a(y)\| \leq \left(2\Cb \bfun(x) \right)^{\frac{k+\alpha - i}{k+\alpha}} \leq 2\Cb \bfun(x)^{\frac{k+\alpha - i}{k+\alpha}}\,,
    \end{equation}
    which applied for $i = \ell+1$, together with~\eqref{eq:0907-1} gives, for all $y \in B(x, 2\ell !\tildec^{\dnum_{\ell}}\bfun(x)^{\frac{1}{k+\alpha}})$, that
    \begin{equation*}
        \|D^\ell a(x) - D^\ell a(y)\| \leq 2\ell !\tildec^{\dnum_{\ell}}\bfun(x)^{\frac{1}{k+\alpha}}2\Cb \bfun(x)^{\frac{k+\alpha - \ell - 1}{k+\alpha}} \leq \tildec^{(1-\alpha)\dnum_{\ell}}\|D^\ell a(x)\| \leq \frac{1}{2}\|D^\ell a(x)\|\,,
    \end{equation*}
    where in the last inequality we used the choice of $\tildec$. We thus established~\eqref{eq:0907-2}.

    Let us now take a vector $v \in \rn$ such that $|v| = 1$ and $|D^\ell a(x)[v^{\ell}]| = \|D^\ell a(x)\|$. We shall prove that $D^\ell a(x)[v^{\ell}] = \|D^\ell a(x)\|$. Suppose that $D^\ell a(x)[v^{\ell}] = -\|D^\ell a(x)\|$. Then, there exists $\xi \in [x, x + 2\ell !\tildec^{\dnum_{\ell}}\bfun(x)^{\frac{1}{k+\alpha}}v]$ such that
    \begin{align}\label{eq:0907-5}
        &a(x + 2\ell !\tildec^{\dnum_{\ell}}\bfun(x)^{\frac{1}{k+\alpha}}v) = \sum_{i=0}^{\ell-1} \frac{(2\ell !)^i\tildec^{i\dnum_{\ell}}\bfun(x)^{\frac{i}{k+\alpha}}(D^ia(x)[v^i])}{i!} + \frac{(2\ell !)^{\ell}\tildec^{\dnum_{\ell}}\bfun(x)^{\frac{\ell}{k+\alpha}}(D^\ell a(\xi)[v^{\ell}])}{2\ell !} \nonumber\\
        &= \sum_{i=0}^{\ell-1} \frac{(2\ell !)^i\tildec^{i\dnum_{\ell}}\bfun(x)^{\frac{i}{k+\alpha}}(D^ia(x)[v^i])}{i!} + \frac{(2\ell !)^{\ell}\tildec^{\dnum_{\ell}}\bfun(x)^{\frac{\ell}{k+\alpha}}(D^\ell a(x)[v^{\ell}])}{\ell !} + \frac{(2\ell !)^{\ell}\tildec^{\dnum_{\ell}}\bfun(x)^{\frac{\ell}{k+\alpha}}((D^\ell a(\xi)-D^\ell a(x))[v^{\ell}])}{\ell !}\,.
    \end{align}
    Equality~\eqref{eq:0907-5} together with the definition of $\ell$ and~\eqref{eq:0907-2} gives
    \begin{align*}
        a(x + 2\ell !\tildec^{\dnum_{\ell}}\bfun(x)^{\frac{1}{k+\alpha}}v) &\leq \sum_{i=0}^{\ell-1} \frac{(2\ell !)^{i}\tildec^{i\dnum_\ell+\alpha \dnum_i}\bfun(x)4\Cb i!}{i!} - \frac{(2\ell !)^{\ell} \tildec^{d_l}\bfun(x)^{\frac{\ell}{k+\alpha}}\|D^\ell a(x)\|}{2\ell !}\\
        &\leq \sum_{i=0}^{\ell-1} (2\ell !)^{\ell-1}\tildec^{i\dnum_\ell+\alpha \dnum_i}\bfun(x)4\Cb  - \frac{(2\ell !)^{\ell+1}\tildec^{(\ell+\alpha)\dnum_{\ell}}\bfun(x)4\Cb }{2\ell !}\\ 
        &= 4\Cb \tildec^{(\ell+\alpha)\dnum_{\ell}}(2\ell !)^{\ell-1}\bfun(x)\left(\sum_{i=0}^{\ell-1} \tildec^{i\dnum_\ell+\alpha \dnum_i - (\ell+\alpha)\dnum_{\ell}} - \ell ! \right)\,.
    \end{align*}
    By the definition of $\dnum_{\ell}$, we have $i\dnum_\ell+\alpha \dnum_i - (\ell+\alpha)\dnum_{\ell} > 0$ for all $i \in \{0, 1, \dots, \ell\}$. Therefore, from the last display and the fact that $a$ is non-negative, we infer
    \begin{equation*}
        0 \leq a(x + 2\ell !\tildec^{\dnum_{\ell}}\bfun(x)^{\frac{1}{k+\alpha}}v) < 4\Cb \tildec^{(\ell+\alpha)\dnum_{\ell}}(2\ell !)^{\ell-1}\bfun(x)\left(l - \ell ! \right) \leq 0\,.
    \end{equation*}
    We get a contradiction, so $D^\ell a(x)[v^{\ell}] = \|D^\ell a(x)\|$.
    
    Note that $2\ell !\tildec^{\dnum_{\ell}}\bfun(x)^{\frac{1}{k+\alpha}} \geq \tildec^{\dnum_k}\bfun(x)^{\frac{1}{k+\alpha}} \geq \sqrt{N}R_x$, from the choice of $\tildec$. Thus,~\eqref{eq:0907-2} and~\eqref{eq:0907-1} imply that 
    \begin{equation*}
       D^\ell a(y)[v^{\ell}] \geq \frac{1}{2}D^\ell a(x)[v^{\ell}] \geq 2\Cb \ell !\tildec^{\alpha \dnum_{\ell}}\bfun(x)^{\frac{k+\alpha-\ell }{k+\alpha}}\,.
    \end{equation*}
    On the other hand,~\eqref{eq:0907-6} implies that $D^\ell a(y)[v^{\ell}] \leq 2\Cb  \bfun(x)^{\frac{k+\alpha-\ell }{k+\alpha}}$.

    We may thus apply Lemma~\ref{lem:Muck} with $\Mbound= 2\Cb \bfun(x)^{\frac{k+\alpha-\ell }{k+\alpha}}$ and $\ve = \ell !\tildec^{\alpha \dnum_{\ell}}$ to get that
    \begin{equation*}
        \left( \dashint_{B(x, R_x)} a^{-\frac{1}{k+\alpha}}\d z \right)^{k+\alpha} \leq C\left(2\Cb  \bfun(x)^{\frac{k+\alpha-\ell }{k+\alpha}} \right)^{-1}R_x^{-\ell}\, \leq C\tildec^{-\ell}\bfun(x)^{-1}\,,
    \end{equation*}
    which means that for some constant $c = c(a, \Omega, \Omprim)$, we have
    \begin{equation}\label{eq:med-alm}
        \dashint_{B(x, R_x)} a^{-\frac{1}{k+\alpha}}\d z \leq cR_x^{-1}\,.
    \end{equation}
    Using~\eqref{eq:med-alm} in conjunction with~\eqref{eq:0707-1} and the fact that $a = \bfun \wfun$, we get
    \begin{equation}\label{eq:med}
        \dashint_{B(x, R_x)} \wfun^{-\frac{1}{k+\alpha}}\d z \leq c\bfun(x)^{\frac{1}{k+\alpha}}R_x^{-1} = c(a, \Omega, \Omprim)\,.
    \end{equation}

    {\bf Step 3: Muckenhoupt condition on large balls.} Let us take any ball $B = B(x, r) \subseteq \Omega$ such that $B$ is not included in $B(x_0, R_{x_0})$ for any $x_0 \in \Omega$. Let us denote
    \begin{equation*}
        G \coloneqq B \cap \{x: \|D^ia(x)\| \neq 0 \text{ for some $i \geq 0$} \}\,.
    \end{equation*}
    By Vitali Covering Theorem there exists a sequence $\{x_n\}_n \subset G$ such that balls $\{B(x_n, \frac{1}{5}R_{x_n})\}_n$ are pairwise disjoint and $G$ is covered by $\{B_n\}_n \coloneqq \{B(x_n, R_{x_n})\}_n$.
    Observe that
    \begin{equation}\label{eq:2507}
    |\tfrac15 B_n| \leq 2^N|\left(\tfrac15 B_n\right) \cap B| \text{ for every $n$.}
    \end{equation}
    Indeed, since $x_n \in B$ but $B$ is not included in $\tfrac{1}{5}B_n$, it holds that $B \cap \tfrac15 B_n$ includes a ball of radius $\tfrac{1}{10}R_n$, from which~\eqref{eq:2507} follows.
    
    Observe that we can apply~\eqref{eq:med} to all $B_n$, which in conjunction with~\eqref{eq:2507} gives
    \begin{align*}
        \int_{G} \wfun^{-\frac{1}{k+\alpha}}\d z \leq \sum_{n=1}^{\infty} \int_{B_n}  \wfun^{-\frac{1}{k+\alpha}} \d z \leq \sum_{n=1}^{\infty} c|B_n| \leq c5^N\sum_{n=1}^{\infty} |\tfrac{1}{5}B_n| \leq \frac{c5^N}{2^N}\sum_{n=1}^{\infty} |(\tfrac{1}{5}B_n) \cap B| \leq \frac{c5^N}{2^N}|B|\,,
    \end{align*}
    where in the last inequality we used that family $\{\tfrac{1}{5}B_n\}_{n}$ is pairwise disjoint. Therefore, as $\wfun$ is equal to $1$ outside $G$, we have
    \begin{equation*}
        \int_{B} \wfun^{-\frac{1}{k+\alpha}} \leq |B \setminus G| + c|B| \leq c|B|\,, 
    \end{equation*}
    and, as $\wfun(\cdot) \leq 1$, it holds
    \begin{equation*}
        \left(\dashint_{B} \wfun \d z\right) \left( \dashint_{B} \wfun^{-\frac{1}{k+\alpha}} \d z\right)^{k+\alpha} \leq c\,.
    \end{equation*}
\end{proof}
\begin{proof}[Proof of Theorem~\ref{theo:decomp}]
    We combine Lemma~\ref{lem:Zpart} with Proposition~\ref{prop:Muck}.    
\end{proof}
\section{Proofs of Theorems~\ref{theo:Cka},~\ref{theo:Lavdecomp} and~\ref{theo:genapprox}}\label{sec:proofs2}
\subsection{Proof of Theorem~\ref{theo:genapprox}}
The proof is essentially based on a well-established technique of convolution with squeezing, employed, for instance, in~\cite{BCFM, BCM, bgs-arma}. The proof of Theorem~\ref{theo:genapprox} is built upon the proof of~\cite[Theorem 3]{BCFM}, with the additional steps taking advantage of the Muckenhoupt part of the weight. Namely, we use the fact that the maximal operator is bounded on the weighted Sobolev space built upon a weight satisfying the Muckenhoupt condition.\newline

We shall briefly recall the results on boundedness of the maximal operator in weighted Lebesgue spaces. For every open and bounded set $\anyopen$, we denote by $M_{\anyopen}$ the maximal operator associated to the set $\anyopen$, defined as
\begin{equation*}
    M_{\anyopen} f(x) \coloneqq \sup_{B : x \in B, 3B \subset \anyopen} \dashint_{B} |f|\d y\,,
\end{equation*}
where the supremum is taken over all balls $B$ containing $x$ and such that $3B \subseteq \anyopen$.

For a given $\wfun \in A_r(\anyopen)$, $r \in (1, \infty)$, we define weighted Lebesgue space as
\begin{equation*}
    L^r(\anyopen, \wfun) \coloneqq \left\{ f : \anyopen \to \R \text{ measurable, } \int_{\Omega} |f(x)|^r\wfun(x)\d x < \infty \right\}\,,
\end{equation*}
with norm defined as
\begin{equation*}
    \|f\|_{L^r(\Omega, \wfun)} \coloneqq \left( \int_{\Omega} |f(x)|^r\wfun(x)\d x \right)^{\frac{1}{r}}\,.
\end{equation*}

From~\cite[Theorem 9]{extrap}, we know that
\begin{equation}\label{eq:maxbound}
    \wfun \in A_r(\anyopen) \implies M_{\anyopen} \text{ is bounded on } L^r(\anyopen, w)\,.
\end{equation}

In our approximation, we employ a convolution with shrinking. Let us fix an open set $\Omega \subseteq \rn$, which is star-shaped with respect to a ball $B(x_0, R)$. We fix $m \in \N$ and $\delta > 0$. Let $\rho_{\delta}$ be a standard regularizing kernel on $\rn$, i.e., $\rho_\delta(x) = \tfrac{1}{\delta^n}\rho(x/\delta)$, where $\rho \in C_c^{\infty}(\rn)$, $\supp \rho \Subset B(0, 1)$, $0 \leq \rho \leq 1$, and $\|\rho\|_{L^1} = 1$. Let us define $\kappa_{\delta} = 1 - \frac{\delta}{R}$. For any measurable function $f : \rn \to \Rm$, we define function $S_{\delta}^{\Omega}f : \rn \to \Rm$ via
\begin{equation*}
    S_{\delta}^{\Omega}f(x) \coloneqq \int_{B(x, \delta)} \rho_{\delta}(x-y)f \left(x_0 + \frac{y-x_0}{\kappa_\delta} \right)\d y = \int_{B_{\delta}(0)} \rho_{\delta}(y)f\left(x_0 + \frac{x - y - x_0}{\kappa_{\delta}} \right)\d y\,.
\end{equation*}
By elementary calculations, if $f$ vanishes outside $\Omega$, then $S_{\delta}^{\Omega} f \in C_c^{\infty}(\Omega)$ for $\delta \in (0, R/4)$. Moreover, for any $u \in W^{1,1}_0(\Omega)$, we have
\begin{equation*}
    \nabla S_{\delta}^{\Omega} u = \frac{1}{\kappa_{\delta}}S_{\delta}^{\Omega}(\nabla u)\,.
\end{equation*}
We shall use the following properties of $S_{\delta}^{\Omega}$.
\begin{lemma}[Lemma~3.1 in~\cite{BCM}]\label{lem:Sd-1}
    If $f \in L^1(\Omega)$, then $S_{\delta}f$ converges to $f$ in $L^1(\Omega)$, and so in measure, as $\delta \to 0$. 
\end{lemma}
\begin{lemma}[Lemma~3.3 in~\cite{BCM}]\label{lem:Sd-2}
Let $u \in W_0^{1, 1}(\Omega)$.
\begin{enumerate}[i)]
    \item If $u \in L^\infty(\Omega)$, then
    \begin{equation}\label{eq:inqLinf}
        \|\nabla S_{\delta}u\|_{L^{\infty}} \leq \delta^{-1}\|u\|_{L^{\infty}}\|\nabla \rho\|_{L^1}\,.
    \end{equation}
    \item If $u \in C^{0, \gamma}(\Omega)$, then
    \begin{equation}\label{eq:Cgamma}
        \|\nabla S_{\delta}u\|_{L^{\infty}} \leq \frac{\delta^{\gamma-1}}{\kappa_{\delta}^{\gamma}}[u]_{C^{0, \gamma}}\|\nabla \rho\|_{L^1}\,.
    \end{equation}
\end{enumerate}
\end{lemma}

We are now in a position to prove Theorem~\ref{theo:genapprox}.
\begin{proof}[Proof of Theorem~\ref{theo:genapprox}]
Given any $u \in W^{1, 1}_0(\Omega)$, one can use a standard truncation argument to approximate $u$ by a sequence $\{u_n\}_n \subset W^{1, 1}_0(\Omega) \cap L^{\infty}(\Omega)$ such that~\eqref{eq:theo2} holds. Indeed, it is sufficient to combine the Lebesgue Monotone Theorem with the fact that the gradient of the truncation at every point coincides with either $0$ or the gradient of the original function. With the use of diagonal argument, we can thus assume that $u \in L^{\infty}(\Omega)$ from now on.

Let us assume that $\Omega$ is a star-shaped domain with respect to the ball $B_R(0)$ and denote $u_{\delta} \coloneqq S_{\delta}^{\Omega}u$. By properties of $S_{\delta}^{\Omega}$, $u_{\delta} \in W^{1, \infty}_0(\Omega)$ for sufficiently small $\delta$. From Lemma~\ref{lem:Sd-1}, we also have that $\|u_{\delta} - u\|_{W^{1, 1}(\Omega)} \xrightarrow{\delta \to 0} 0$. We shall moreover prove that
\begin{equation*}
    \left\{ |\nabla u_{\delta}|^p + a|\nabla u_{\delta}|^q \right\}_{\delta} \text{ is equiintegrable.}
\end{equation*}
We shall work on both cases i) and ii) simultaneously. In case of i), we shall denote $\gamma = 0$ and momentarily also assume that $p \leq N$. Consequently, in both cases, the condition on $p$ and $q$ reads $q \leq p + \frac{k+\alpha}{1-\gamma}$. We further denote $\vartheta \coloneqq (q-p)(1-\gamma)$ and $\tau \coloneqq \frac{2 + \diam \Omega}{R}$. Then, for any $\delta \leq \tfrac{R}{2}$, $x \in \Omega$ and $y \in B_{\delta}(0)$, we have
\begin{equation}\label{eq:1007-1}
    \left| \frac{x-y}{\kappa_{\delta}} - x \right| \leq \frac{|y|}{\kappa_{\delta}} + \frac{|x|(1-\kappa_{\delta})}{\kappa_{\delta}} \leq \frac{2\delta}{R} + \frac{(\diam \Omega)\delta}{R} = \tau\delta\,. 
\end{equation}
Using that $\bfun \in \SP^{\vartheta}(\Omega)$, we have that there exists a constant $\Cb $ such that whenever $|x-y| \leq \tau\delta$, it holds
\begin{equation}\label{eq:1007-2}
    \bfun(x) \leq C_{\bfun}(\bfun(y) + \tau^{\vartheta} \delta^{\vartheta})\,.
\end{equation}
Let us now fix any $x \in \Omega$ and $\delta \leq \tfrac{R}{2}$. Let
\begin{equation*}
    \bfun^-_{\delta} \coloneqq \inf_{y \in B_{\delta}(0)} \bfun \left( \frac{x-y}{\kappa_{\delta}} + x \right)\,.
\end{equation*}
By~\eqref{eq:1007-1} and~\eqref{eq:1007-2}, we have
\begin{equation*}
    \bfun(x) \leq \Cb (\bfun^-_{\delta} + \tau^{\vartheta}\delta^{\vartheta})\,.
\end{equation*}
Therefore, we can estimate
\begin{equation}\label{eq:1007-7}
    a(x)|\nabla u_{\delta}(x)|^q = \bfun(x)\wfun(x)|\nabla u_{\delta}(x)|^q \leq \bfun^-_{\delta}\wfun(x)|\nabla u_{\delta}(x)|^q + \|\wfun\|_{L^{\infty}(\Omega)}\tau^{\vartheta}\delta^{\vartheta}|\nabla u_{\delta}(x)|^q\,.
\end{equation}
By Lemma~\ref{lem:Sd-2}, there exists a constant $c$, depending only on $\rho$ and $u$, such that
\begin{equation*}
    \|\wfun\|_{L^{\infty}(\Omega)}\tau^{\vartheta}\delta^{\vartheta}|\nabla u_{\delta}(x)|^q \leq \|\wfun\|_{L^{\infty}}\tau^{\vartheta}\delta^{\vartheta}\|\nabla u_{\delta}\|_{L^{\infty}}^{q-p} |\nabla u_{\delta}(x)|^p \leq c\tau^{\vartheta} |\nabla u_{\delta}(x)|^p\,.
\end{equation*}
Note that, by Jensen's inequality, it holds
\begin{equation}\label{eq:1007-9}
    |\nabla u_{\delta}(x)|^p = \left( \int_{B_{\delta}(0)} \left| \nabla u \left( \frac{x - y}{\kappa_{\delta}} \right) \right|\rho_{\delta}(y)\d y \right)^p \leq \int_{B_{\delta}(0)} \left| \nabla u \left( \frac{x - y}{\kappa_{\delta}} \right) \right|^p\rho_{\delta}(y)\d y = S_{\delta}^{\Omega}(|\nabla u|^p)\,.
\end{equation}
As $\cF[u] < \infty$, we have $|\nabla u|^p \in L^1(\Omega)$, so by Lemma~\ref{lem:Sd-1} and Vitali Convergence Theorem, we get equiintegrability of the family $\{S_{\delta}^{\Omega}(|\nabla u|^p)\}_{\delta}$. By~\eqref{eq:1007-9}, we have that
\begin{equation}\label{eq:1007-5}
    \text{the family } \left\{ |\nabla u_{\delta}|^p \right\}_{\delta} \text{ is equiintegrable.}
\end{equation}
On the other hand, by the definition of $u_{\delta}$ and $\bfun^-_{\delta}$, we have
\begin{align}\label{eq:1007-4}
    \bfun^-_{\delta}\wfun(x)|\nabla u_{\delta}(x)|^q &\leq \bfun^-_{\delta}\wfun(x)\left(\int_{B_{\delta}(0)} \left| \nabla u \left( \frac{x-y}{\kappa_{\delta}} \right)\right|\rho_{\delta}(y)\d y\right)^q \nonumber \\
    &\leq \wfun(x)\left(\int_{B_{\delta}(0)} \bfun^{\frac{1}{q}}\left(\frac{x-y}{\kappa_{\delta}} \right)\left| \nabla u \left( \frac{x-y}{\kappa_{\delta}} \right)\right|\rho_{\delta}(y)\d y\right)^q = \wfun(x) \left(S_{\delta}^{\Omega}\left(\bfun^{\frac{1}{q}}|\nabla u|\right)(x)\right)^q\,.
\end{align}
Let us denote $\bu \coloneqq \bfun^{1/q}|\nabla u|$. From the definition of $S_{\delta}^{\Omega}$ and $\rho_{\delta} \leq \delta^{-N}$, we get
\begin{align}\label{eq:1107-1}
    S_{\delta}^{\Omega}f(x) &= \int_{B(x, \delta)} \rho_{\delta}(x-y)f \left(\frac{y}{\kappa_\delta} \right)\d y \leq \frac{1}{\delta^N}\int_{B(x, \delta)} f \left(\frac{y}{\kappa_\delta} \right)\d y = \frac{\kappa_{\delta}^N}{\delta^N} \int_{B(x/\kappa_{\delta}, \delta/\kappa_{\delta})} f(y)\d y\,.
\end{align}
From~\eqref{eq:1007-1}, we infer that $|x - x/\kappa_{\delta}| \leq \tau\delta$ for sufficiently small $\delta$. Thus, for sufficiently small $\delta$, we have
\begin{equation*}
    B\left( \frac{x}{\kappa_{\delta}}, \frac{\delta}{\kappa_{\delta}} \right) \subseteq B\left( \frac{x}{\kappa_{\delta}}, 2\delta \right) \subseteq B\left(x, (\tau + 2)\delta \right)\,.
\end{equation*}
In turn, by~\eqref{eq:1107-1}, we have
\begin{equation}\label{eq:1107-2}
    S_{\delta}^{\Omega}f(x) \leq \frac{\kappa_{\delta}^N}{\delta^N} \int_{B(x, (\tau + 2)\delta)} f(y)\d y \leq\frac{2^N(\tau + 2)^N}{|B_1|} \dashint_{B(x, (\tau + 2)\delta)} f(y)\d y
\end{equation}

Recall that $\wfun \in A_q(\Uset)$ and $\Omega \Subset \Uset$. Taking any $\delta$ such that~\eqref{eq:1107-2} holds and additionally $\delta \leq \frac{3\dist(\Omega, \partial \Uset)}{\tau + 2}$, we get
\begin{equation*}
    S_{\delta}^{\Omega}f(x) \leq cM_{\Uset}f(x)\,, 
\end{equation*}
where $c = \frac{2^N(\tau + 2)^N}{|B_1|}$. Note that by~\eqref{eq:maxbound} we have that $\wfun (M_{\Uset}f)^q \in L^1(\Omega)$. From~\eqref{eq:1007-4}, we obtain that
\begin{equation}\label{eq:1007-6}
    \text{the family } \left\{ \bfun^-_{\delta}\wfun|\nabla u_{\delta}| \right\}_{\delta} \text{ is equiintegrable.}
\end{equation}
From the estimate~\eqref{eq:1007-7} and lines~\eqref{eq:1007-5} and~\eqref{eq:1007-6}, we get that the family $\{|\nabla u_{\delta}|^p + a|\nabla u_{\delta}|^q\}_\delta$ is equiintegrable. By Lemma~\ref{lem:Sd-1}, we also have that $\nabla u_{\delta}$ converges to $\nabla u$ in measure, which by Vitali Convegence Theorem means that $\cF[u_{\delta}] \xrightarrow{\delta \to 0} \cF[u]$. This finishes the proof in the case of $\Omega$ being a star-shaped domain with respect to a ball centred in $0$ and, in the case i), with the additional assumption that $p \leq N$.

If one starts from $\Omega$, which is a star-shaped domain with respect to the ball not centred in $0$, one can simply translate the whole problem so that $\Omega$ is star-shaped with respect to a ball centred in $0$. \newline

Now we shall focus on the case of $\Omega$ being an arbitrary bounded Lipschitz domain. From~\cite[Lemma~8.2]{C-b}, we infer that set $\overline{\Omega}$ can be covered by a finite family of sets $\{U_i\}_{i=1}^{m}$ such that each $\Omega_i :=\Omega\cap U_i$ is a star-shaped domain with respect to some ball, for all $i = 1, \dots, m$. Then $\Omega=\bigcup_{i=1}^{m}\Omega_i\,$.    By~\cite[Proposition 2.3, Chapter 1]{necas}, there exists the partition of unity related to the partition $\{U_i\}_{i=1}^{m}$, i.e., the family $\{\theta_i\}_{i=1}^{m}$ such that
\begin{equation*}
0\le\theta_i\le 1,\quad\theta_i\in C^\infty_c(U_i),\quad \sum_{i=1}^{m}\theta_i(x)=1\ \ \text{for}\ \ x\in\Omega\,.
\end{equation*}
By the previous paragraph, for every $i = 1, 2, \dots, m$, as $\Omega_i$ is a star-shaped domain with respect to some ball, $\cF[\theta_i u;\Omega_i] < \infty$, and $\theta_i u \in W^{1, 1}_0(\Omega)$, there exist a sequence $\{u_{\delta}^i\}_{\delta}$ such that $\|u_{\delta}^i - \theta_i u\|_{W^{1, 1}(\Omega)} \xrightarrow{\delta \to 0} 0$ and $\cF[u_{\delta}^i; \Omega] \xrightarrow{\delta \to 0} \cF[\theta_iu;\Omega]$. Let us now consider the sequence $(I_\delta)_\delta$ defined as
\begin{equation*}I_{\delta} \coloneqq \sum_{i=1}^{m} u^i_{\delta}\,.
\end{equation*}
As we have that $u_{\delta}^i \to \theta_iu$ in $W^{1, 1}(\Omega)$ for every $i$, we have $I_{\delta} \to u$ in $W^{1, 1}(\Omega)$. Since the sequence $(\nabla  u_{\delta}^i)_\delta$ converges to $\nabla (\theta_i u)$ in measure and $\sum_{i=1}^{m} \nabla(\theta_i u) = \nabla u$, it holds that
\begin{equation}\label{eq:jan9}
    \left(\nabla I_{\delta}\right)_{\delta} \xrightarrow{\delta \to 0} %\text{ converges to }
    \nabla u \text{ in measure.}
\end{equation}
Moreover, for any $x \in \Omega$ we have that
\begin{align}\label{eq:jan11}
    \left| \nabla I_{\delta}(x) \right|^p + a(x)\left| \nabla I_{\delta}(x) \right|^q &\leq \sum_{i=1}^{m} \left(m^{p-1}|\nabla(u_{\delta}^i)(x)|^p + m^{q-1}a(x)|\nabla(u_{\delta}^i)(x)|^q\right)\nonumber\\
    &\leq m^{q-1}\sum_{i=1}^{m}  \left(|\nabla(u_{\delta}^i)(x)|^p + a(x)|\nabla(u_{\delta}^i)(x)|^q\right)\,.
\end{align}
As for all $i = 1, 2, \dots, m$, we have that $\cF[u^i_{\delta}]$ converges, it holds that the family\\ $\{|\nabla(u_{\delta}^i)|^p + a|\nabla(u_{\delta}^i)|^q\}_{\delta}$ 
is equiintegrable. Therefore, the estimate~\eqref{eq:jan11} gives us that\[\text{ the family $\quad\left\{ \left| \sum_{i=1}^{m} \nabla u_{\delta}^i\right|^p + a \left| \sum_{i=1}^{m} \nabla u_{\delta}^i\right|^q\right \}_{\delta}\quad$ is equiintegrable. } \]
This, together with~\eqref{eq:jan9} and Vitali Convergence Theorem, as well as the fact that  $I_{\delta} \to u$ in $W^{1, 1}(\Omega)$, gives us the result for an arbitrary bounded Lipschitz domain $\Omega$. \newline

To drop the assumption that $p \leq N$ in the case i), we observe that whenever $p > N$, then $\cF[u] < \infty$ implies $u \in W^{1, p}(\Omega)$, which further means that $u \in C^{0, 1 - \frac{N}{p}}(\Omega)$. Already proven case ii) thus gives us~\eqref{eq:theo2} in case of $q \leq p + \frac{\vk p}{N}$, which gives Assertion i).

\end{proof}
\subsection{Proof of Theorems~\ref{theo:Cka} and~\ref{theo:Lavdecomp}}
\begin{proof}[Proof of Theorem~\ref{theo:Lavdecomp}]
    Let $\{u_n\}_n \subset W^{1, 1}_0(\Omega)$ be a sequence satisfying
    \begin{equation*}
        \cF[u_n + u_0;\Omega] \xrightarrow{n \to \infty} \inf_{u \in u_0 + W^{1, 1}_0(\Omega)} \cF[u;\Omega]\,.
    \end{equation*}
    Let us fix $n \in \N$. Note that
    \begin{equation*}
        \cF[u_n;\Omega] \leq 2^q\left(\cF[u_n + u_0;\Omega] + \cF[u_0;\Omega] \right) < \infty\,.
    \end{equation*}
    Thus, we may apply Theorem~\ref{theo:genapprox} to find a sequence $\{\varphi_{n, k}\}_k \subset C_c^{\infty}(\Omega)$ such that
    \begin{equation*}
        \cF[\varphi_{n, k};\Omega] \xrightarrow{k \to \infty} \cF[u_n;\Omega] \quad \text{ and } \quad \|\varphi_{n, k} - u_n\|_{W^{1, 1}(\Omega)} \xrightarrow{k \to \infty} 0\,.
    \end{equation*}
    In particular, the sequence $\{|\nabla \varphi_{n, k}|^p + a|\nabla \varphi_{n, k}|^q\}_k$ is equiintegrable. As we have
    \begin{equation*}
        |\nabla (\varphi_{n, k} + u_0)|^p + a|\nabla (\varphi_{n, k} + u_0)|^q \leq 2^q \left( |\nabla \varphi_{n, k}|^p + a|\nabla \varphi_{n, k}|^q + |\nabla u_0|^p + a|\nabla u_0|^q \right)\,,
    \end{equation*}
    the sequence $\{|\nabla (\varphi_{n, k} + u_0)|^p + a|\nabla (\varphi_{n, k} + u_0)|^q\}_k$ is also equiintegrable. Moreover, it converges in measure to $|\nabla (u_n + u_0)|^p + a|\nabla (u_n + u_0)|^q$. This means that
    \begin{equation*}
        \cF[\varphi_{n, k} + u_0;\Omega] \xrightarrow{k \to \infty} \cF[u_n + u_0;\Omega]\,
    \end{equation*}
    Therefore,
    \begin{equation*}
        \inf_{u \in u_0 + W^{1, 1}_0(\Omega)} \cF[u;\Omega] \leq\inf_{u \in u_0 + C_c^{\infty}(\Omega)} \cF[u;\Omega] \leq \inf_{n, k} \cF[\varphi_{n, k} + u_0;\Omega] \leq \inf_{n} \cF[u_n + u_0;\Omega] \leq \inf_{u \in u_0 + W^{1, 1}_0(\Omega)} \cF[u;\Omega]\,,
    \end{equation*}
    which finishes the proof.
\end{proof}
\begin{proof}[Proof of Theorem~\ref{theo:Cka}]
    Let $q = p + (\ell + \beta)\max(1, p/N)$, where $\ell \in \N$ and $\beta \in (0, 1]$. From assumption $q \leq p + (k + \alpha)\max(1, p/N)$, we have that either $\ell < k$ or $\ell = k$ and $\beta \leq \alpha$. In particular, $a \in C^{\ell, \beta}(\Uset)$. 

    Let $\Omprim$ be any open set such that $\Omega \Subset \Omprim \Subset \Uset$. We apply Theorem~\ref{theo:decomp} to function $a \in C^{\ell, \beta}(\Uset)$ and $\Omprim$, obtaining functions $\bfun \in \mathcal{Z}^{\ell+ \beta}(\Omprim)$ and $\wfun \in A_{\ell+ 1 + \beta}(\Omprim)$. Note that $q \geq p + \ell+ \beta \geq \ell+ 1 + \beta$, so $\wfun \in A_q(\Omprim)$, as $A_{\ell+ 1 + \beta}(\Omprim) \subseteq A_q(\Omprim)$. Applying Theorem~\ref{theo:Lavdecomp}, we obtain the conclusion. 
\end{proof}
%
%\section*{Appendix}
%\renewcommand\thesection{\Alph{section}}
%\renewcommand{\theHsection}{\Alph{section}}
%\setcounter{section}{0}
%\setcounter{proposition}{0}
%\section{Supplementary lemmas}
%
\section*{Acknowledgement}
The author would like to thank Iwona Chlebicka (University of Warsaw) for several suggestions on how to improve the overall quality of the paper.
\printbibliography
\end{document}